\documentclass{amsart}
\usepackage{graphicx}
\usepackage{amssymb,amscd,amsthm,amsxtra}
\usepackage{latexsym}
\usepackage{epsfig}
\usepackage{mathtools}

\newtheorem{thm}{Theorem}[section]
\newtheorem{cor}[thm]{Corollary}
\newtheorem{lem}[thm]{Lemma}
\newtheorem{prop}[thm]{Proposition}
\theoremstyle{definition}
\newtheorem{defn}[thm]{Definition}
\theoremstyle{remark}
\newtheorem{rem}[thm]{Remark}
\numberwithin{equation}{section}

\newcommand{\be}{\begin{equation}}
\newcommand{\ee}{\end{equation}}

\newcommand{\R}{\mathbb R}

\newcommand{\eps}{\epsilon}

\newcommand{\p}{\partial}

\newcommand{\comment}[1]{}

\begin{document}

\title[Higher regularity]{
Regularity of higher order in two-phase
free boundary problems}
\author{Daniela  De Silva}
\address{Department of Mathematics, Barnard College, Columbia University, New York, NY 10027}
\email{\tt  desilva@math.columbia.edu}
\author{Fausto Ferrari}
\address{Dipartimento di Matematica dell'Universit\`a di Bologna, Piazza di Porta S. Donato, 5, 40126 Bologna, Italy.}
\email{\tt fausto.ferrari@unibo.it}
\author{Sandro Salsa}
\address{Dipartimento di Matematica del Politecnico di Milano, Piazza Leonardo da Vinci, 32, 20133 Milano, Italy.}
\email{\tt sandro.salsa@polimi.it }
\begin{abstract}We develop further our strategy in \cite{DFS} to show that flat or Lipschitz free boundaries of two-phase problems
with forcing terms are locally $C^{2,\gamma }.$ \end{abstract}


\maketitle

\section{Introduction}\label{introduzione}

Let $f_{\pm }\in C^{0,\gamma }(B_{1}),$ where $B_1$ denotes the unit ball in 
$\mathbb{R}^n,$ $n\geq 2,$ centered at $0$, and consider the two-phase
problem 
\begin{equation}
\left\{ 
\begin{array}{ll}
\Delta u=f_+ & \hbox{in $B_1^+(u),$} \\ 
\  & \  \\ 
\Delta u=f_{-} & \hbox{in $B_1^-(u),$} \\ 
\  & \  \\ 
u_{\nu }^{+}=G(u_{\nu }^{-}) & 
\hbox{on $F(u):= \partial
B^+_1(u) \cap B_1.$} \\ 
& 
\end{array}%
\right.  \label{fbtp}
\end{equation}%
Here 
\begin{equation*}
B_{1}^{+}(u):=\{x\in B_{1}:u(x)>0\},\quad B_{1}^{-}(u):=\{x\in
B_{1}:u(x)\leq 0\}^{\circ },
\end{equation*}%
while $u_{\nu }^{+}$ and $u_{\nu }^{-}$ denote the normal derivatives in the
inward direction to $B_{1}^{+}(u)$ and $B_{1}^{-}(u)$ respectively. The
function $G:\mathbb{[}0,\infty )\rightarrow \mathbb{R}^{+}$ satisfies the
usual ellipticity assumption: 
\begin{equation}  \label{G_e}
\begin{split}
G\quad & \mbox{is strictly increasing,}\:\: \:\: G(0)>0,\:\: \mbox{and}%
\:\:G(b)\rightarrow \infty\:\:\mbox{as}\:\: b\rightarrow \infty .
\end{split}%
\end{equation}
For simplicity, we assume that $G\in C^{2}(\mathbb{[}0,\infty ))$ and say $%
G(0)=1.$

Typical examples of inhomogeneous two-phase problems are the
Prandtl-Bachelor model in fluid-dynamics (see e.g. \cite{B1,EM}), or the
eigenvalue problem in magneto-hydrodynamics considered in \cite{FL}. Other
examples come from limits of singular perturbation problems with forcing
term as in \cite{LW}, where the authors analyze solutions to \eqref{fbtp},
arising in the study of flame propagation with nonlocal effects. \smallskip

Our main result gives $C^{2,\gamma ^{\ast }}$ regularity of flat free
boundaries.
Precisely, we prove the following theorem, where we call universal any constant depending on $n,\gamma ,L:= Lip(u),\Vert f_{\pm }\Vert
_{C^{0,\gamma }(B_{1})},$ and $\Vert G\Vert _{C^{2}([0,L+1])}$.

\begin{thm}
\label{flatmain1} Let $u$ be a (Lipschitz) viscosity solution to \eqref{fbtp}
in $B_{1}$. There exists a universal constant $\bar{\eta}>0$ such that, if 
\begin{equation}
\{x_{n}\leq -\eta \}\subset B_{1}\cap \{u^{+}(x)=0\}\subset \{x_{n}\leq \eta
\},\quad \text{for $0\leq \eta \leq \bar{\eta},$ }  \label{flat}
\end{equation}%
then $F(u)$ is $C^{2,\gamma ^{\ast }}$ in $B_{1/2}$ for a small $\gamma
^{\ast }$ universal, with the $C^{2,\gamma ^{\ast }}$ norm bounded by a
universal constant. 
\end{thm}

In view of Theorem 1.3 in \cite{DFS}, if \eqref{flat} holds then the free
boundary $F(u)$ is locally $C^{1,\bar{\gamma}}$. Note that the same
conclusion holds if $F\left( u\right) $ is a graph of a Lipschitz function
(see Theorem 1.4 in \cite{DFS}). Therefore, throughout the paper we will
assume that $F(u)$ is $C^{1,\bar{\gamma}}$ and hence $u$ is a classical
solution, i.e. the free boundary condition is satisfied in a pointwise
sense. To fix ideas, let us say that $0<\gamma \leq \bar{\gamma}$.\smallskip

Our result extends without much effort to more general linear uniformly
elliptic equations with $C^{0,\gamma }$
coefficients and to more general free boundary jump conditions $u_{\nu
}^{+}=G(u_{\nu }^{-},\nu ,x)$, where $G$ is $C^{2}$ with respect to all its arguments. For those operators, also
considering the existence paper \cite{DFS2}, the theory of viscosity
solutions to inhomogeneous free boundary problems has reached a considerable
level of completeness. Perhaps, the only relevant open question remains the
analysis of singular (\textquotedblleft nonflat\textquotedblright ) points.

For fully nonlinear operators, we proved in \cite{DFS3} that for a fairly
general class of problems (with right-hand side), Lipschitz viscosity solutions with Lipschitz or
flat (in the sense of \eqref{flat}) free boundaries are indeed classical.
The questions of Lipschitz continuity of solutions and higher regularity of
the free boundary remain open problems.\smallskip

In order to explain the significance of our main theorem, we describe here the state of the art
about the higher regularity theory for two-phase free boundary problems. In the seminal paper \cite{KNS},
the authors use a zero order hodograph transformation and a suitable reflection
map, to locally reduce a two-phase problem to an
elliptic and coercive system of nonlinear equations (see Appendix B).
The existing literature on the regularity of solutions to nonlinear systems developed in 
\cite{ADN, Morrey} can be applied as long as the solution $u$ is $C^{2,\alpha}$ (for some $\alpha>0$) up to the free boundary (from either side). Hence, the following corollary of Theorem \ref{flatmain1} holds.

\begin{cor}
\label{corcor} Let $k$ be a nonnegative integer. Assume that $f_{\pm }\in
C^{k,\gamma }\left( B_{1}\right) $ and $G$ is $C^{2+k}$. Then $F\left(
u\right) \cap B_{1/2}$ is $C^{k+2,\gamma ^{\ast }}.$ If $f_{\pm }$ are $%
C^{\infty }$ or real analytic in $B_{1}$, then $F\left( u\right) \cap
B_{1/2} $ is $C^{\infty }$ or real analytic, respectively.
\end{cor}

As noted in the recent work \cite{KL}, in the case when the governing equation in \eqref{fbtp} is in divergence form the initial assumption to obtain the Corollary above is that $u \in C^{1,\alpha}$. It is not evident that the general case of linear uniformly
elliptic equations with $C^{0,\gamma}$ coefficients can also be treated in a similar manner. On the other hand, the case when the leading operator is say a convex (or concave) fully nonlinear operator definitely requires the solution to have H\"older second derivatives (from both sides). 

Our purpose is to develop a general strategy that would apply to a larger class of problems, possibly to include also the case of fully nonlinear operators.

Other related higher regularity results can be found in \cite{E, K}.


The overall strategy for the proof of Theorem \ref{flatmain1} is based, as
in Theorem 1.3 of \cite{DFS}, on a compactness argument leading to a
limiting linearized problem in which the information for an improvement of
flatness is stored. However, reaching the $C^{2,\gamma}$ regularity requires
a much more involved process because of the possible degeneracy of the
negative part. Indeed this causes a delicate interplay between the two
phases, as we shall try to explain in the next section. Ultimately the main
source of difficulties is due to the presence of a forcing term of general
sign in the negative phase. Indeed, if $f_{-}\geq 0$, Hopf maximum principle
would imply nondegeneracy (also) on the negative side, making the two-phases
of comparable size and considerably simplifying the final iteration
procedure. It is worth noticing, that however even in this easier scenario
(and in particular in the homogeneous case), if one wants to attain uniform
estimates with universal constants, then one must employ the more involved
methods developed here for the degenerate case.

The paper is organized as follows. In Section \ref{strategy} we outline the strategy of
the proof and in Section \ref{initialconfig} through \ref{iteration7} we implement it. In  Appendix A,
we provide a refinement of the classical pointwise $C^{1,\alpha}$ estimates
for elliptic equations. This is a technical tool used in the paper. In Appendix B we sketch the main steps for the reduction of our problem (\ref{fbtp}) to a system  of nonlinear equations as in \cite{KNS}. 


\section{Outline and strategy}\label{strategy}

In this section, we outline the main strategy in the proof of Theorem \ref%
{flatmain1}, trying to emphasize the key points, also in comparison to the 
\emph{flatness implies }$C^{1,\gamma}$\emph{\ case} in \cite{DFS}. The first
thing to do is to reinforce the notion of flatness, tailoring it for the
attainment of $C^{2,\gamma}$ regularity. This can be done by introducing a
suitable class of functions that we call \emph{two-phase} and \emph{one-phase%
} polynomials. In principle second order polynomials should be enough but it
turns out that we need a small third order perturbation.

Given $\omega \in \mathbb{R}^{n},$ with $|\omega |=1$, and let $S_{\omega }$
be an orthonormal basis containing $\omega .$ Let $M\in S^{n\times n}$
satisfy 
\begin{equation*}
M \omega =0
\end{equation*}%
and define 
\begin{equation*}
P_{M,\omega }(x)=x\cdot \omega -\frac{1}{2}x^{T}Mx.
\end{equation*}%
Set, 
\begin{equation*}
V_{M,\omega ,a,b}^{\alpha ,\beta }\left( x\right) =\alpha (1+a\cdot
x)P_{M,\omega }^{+}\left( x\right) -\beta (1+b\cdot x)P_{M,\omega
}^{-}\left( x\right) ,\quad \alpha >0,\beta \geq 0,a,b\in \mathbb{R}^{n}
\end{equation*}%
where the superscripts $\pm $ denote as usual the positive/negative part of
a function. These are our two-phase polynomials, one-phase if $\beta =0$. In
the particular case when $M=0,a=b=0,\omega =e_{n}$ we obtain the two-plane
function: 
\begin{equation*}
U_{\beta }(x)=\alpha x_{n}^{+}-\beta x_{n}^{-}.
\end{equation*}%
The unit vector $\omega $ establishes the \textquotedblleft direction of
flatness\textquotedblright .

We shall need to work with a subclass, strictly related to problem (\ref%
{fbtp}), at least at the origin. We denote by $\mathcal{V}_{f_{\pm },G}$ the
class of functions of the form $V_{M,\omega ,a,b}^{\alpha ,\beta }$ for
which 
\begin{equation*}
2\alpha a\cdot \omega -\alpha tr M=f_{+}(0)
\end{equation*}%
\begin{equation*}
2\beta b\cdot \omega -\beta tr M=f_{-}(0)\quad \text{if $\beta \neq 0,$}
\end{equation*}%
\begin{equation*}
\alpha =G(\beta ), \quad \text{if $\beta \neq 0,$}
\end{equation*}%
and 
\begin{equation*}
\alpha a\cdot \omega ^{\perp }=\beta G^{\prime }\left( \beta \right) b\cdot
\omega ^{\perp },\quad \forall \omega ^{\perp }\in S_{\omega }.
\end{equation*}%
The role of the last condition will be clear in the sequel (e.g. Proposition %
\ref{sub}).

When $\beta =0$, then there is no dependence on $b$ and $a\cdot \omega
^{\perp }=0$. Thus, we drop the dependence on $\beta ,b, G$ and $f_{-}$ in
our notation above and we indicate the dependence on $a_\omega:=a \cdot
\omega.$

We introduce the following definitions.

\begin{defn}
Let $V=V_{M,\omega ,a,b}^{\alpha ,\beta }$. We say that $u$ is $(V,\epsilon
,\delta )$ flat in $B_{1}$ if 
\begin{equation*}
V(x-\epsilon \omega )\leq u(x)\leq V(x+\epsilon \omega )\quad \text{in $%
B_{1} $}
\end{equation*}%
and 
\begin{equation*}
|a|,|b^{\prime }|,\Vert M\Vert \leq \delta \epsilon ^{1/2},\quad |b_{n}|\leq
\delta ^{2},\quad |b_{n}|\Vert M\Vert \leq \delta ^{2}\epsilon .
\end{equation*}
\end{defn}

Given $V=V_{M,\omega ,a,b}^{\alpha ,\beta }$, set 
\begin{equation*}
V_{r}(x)=\frac{V(rx)}{r}
\end{equation*}%
and notice that 
\begin{equation*}
V_{r}=V_{rM,\omega ,ra,rb}^{\alpha ,\beta }.
\end{equation*}

\begin{defn}
Let $V=V_{M,\omega ,a,b}^{\alpha ,\beta }$. We say that $u$ is $(V,\epsilon
,\delta )$ flat in $B_{r}$ if the rescaling 
\begin{equation*}
u_{r}(x):=\frac{u(rx)}{r}
\end{equation*}%
is $(V_{r},\frac{\epsilon }{r},\delta )$ flat in $B_{1}.$
\end{defn}

Notice that if $u$ is $(V,\epsilon ,\delta )$ flat in $B_{r}$ then 
\begin{equation*}
V(x-\epsilon \omega )\leq u(x)\leq V(x+\epsilon \omega )\quad \text{in $%
B_{r} $}.
\end{equation*}

The parameter $\epsilon $ measures the level of polynomial approximation and 
$\delta $ is a flatness parameter (also controlling the $C^{0,\gamma }$
norms of $f_{+}$ and $f_{-}$).

Thus roughly our purpose is to show that $u$ is $(V_{k},\lambda
_{k}^{2+\gamma ^{\ast }},\delta )$ flat in $B_{\lambda _{k}}$ for $\lambda
_{k}=\eta ^{k}$ and all $k\geq 0,$ for some $\delta ,\eta $ small and a
sequence of $V_{k}$ converging to a final profile $V_{0}$. This would give
uniform pointwise $C^{2,\gamma ^{\ast }}$ regularity both for the solution
and the free boundary in $B_{1/2}$.

The starting point is to show (Section \ref{initialconfig}, Lemma \ref{first_step_1}) that the
flatness condition \eqref{flat} allows us to normalize our solution so that
a rescaling $u_{\bar{r}}$ of $u$ falls into one of the following cases, with
suitable $\bar{\lambda},\bar{ \delta}.$ This kind of dichotomy parallels in
a sense what happens in \cite{DFS}.



Case a). $u_{\bar{r}}$ is $(V,\bar{\lambda}^{2+\gamma },\bar{\delta})$ flat
for some $V=V_{0,e_{n},a,b}^{\alpha ,\beta }\in \mathcal{V}_{f_{\pm },G}.$
Moreover, $\beta \bar{\delta}$ controls the $C^{0,\gamma }$ seminorms of $%
f_{-}.$This case corresponds to a nondegenerate configuration, in which the
two phases have comparable size and $u_{\bar{r}}$ is trapped between two
translations of a genuine two-phase polynomial. \smallskip

Case b). $u_{\bar{r}}^{+}$ is $(V,\bar{\lambda}^{2+\gamma },\bar{\delta})$
flat for some $V=V_{0,e_{n},a_{n}}^{1}\in \mathcal{V}_{f_{+}},$ and $u_{\bar{%
r}}^{-}$ is close to a purely quadratic profile $cx_{n}^{2}.$ This case
corresponds to a degenerate configuration, where the negative phase has
either zero slope or a small one (but not negligible) with respect to $u_{%
\bar{r}}^{+}$, and $u_{\bar{r}}^{+}$ is trapped between two translations of
a one-phase polynomial. Note that this situation cannot occur if $f_{-}\geq
0 $, unless $u^-$ is identically zero.\smallskip


Next we examine how the initial flatness corresponding to cases a) and b)
above improves successively at a smaller scale. In Section \ref{improvflat}, we construct
the following \textquotedblleft subroutines\textquotedblright , to be
implemented in the course of the final iteration towards $C^{2,\gamma ^{\ast
}}$ regularity.

\emph{Two-phase }flatness improvement (Proposition \ref{IF2tp}): if $u$ is $%
(V,\bar{\lambda}^{2+\gamma },\bar{\delta})$ flat for some $%
V=V^{\alpha,\beta}_{M,\omega,a,b}\in \mathcal{V}_{f_{\pm },G}$ in $%
B_{\lambda }$, the $C^{0,\gamma }$ seminorms of $f_{+}$ and $f_{-}$ are
controlled by $\bar{\delta}$ and $\beta \bar{\delta}$, respectively, then,
in $B_{\lambda \eta },$ $u$ enjoys a $C^{2,\gamma }$ flatness improvement,
i.e. $u$ is $(\bar{V},\left( \eta \lambda \right) ^{2+\gamma },\bar{\delta})$
flat for some $\bar{V}\in \mathcal{V}_{f_{\pm },G}$, properly close to $V$.

\emph{One-phase }flatness improvement (Proposition \ref{IF}): if $u^{+}$ is $%
(V,\bar{\lambda}^{2+\gamma },\bar{\delta})$ flat for some $%
V=V^\alpha_{M,\omega,a_\omega}\in \mathcal{V}_{f_{+}}$ in $B_{\lambda },$
the $C^{0,\gamma }$ seminorm of $f_{+}$ is controlled by $\bar{\delta}$ and $%
u_{\nu }^{+}$ is close to $\alpha$ on $F\left( u\right) $, then $u^{+}$
enjoys a $C^{2,\gamma }$ flatness improvement, with $\bar{V}\in \mathcal{V}%
_{f_{+}}$, properly close to $V$.\smallskip

The achievement of the improvements above relies on a Harnack inequality,
and a higher order refinement of Theorems 4.1 and 4.4 in \cite{DFS}. This
gives the necessary compactness to pass to the limit in a sequence of
suitable renormalizations of $u$ and obtain a limiting transmission problem
(Neumann problem in the one phase-case). From the regularity of the solution
of this problem we get the information to improve the two-phase or one-phase
approximation for $u$ or $u^{+}$ respectively, and hence their
flatness.\smallskip

Now we start iterating. As we have seen, according to Lemma \ref%
{first_step_1}, after a suitable rescaling, we face a first dichotomy
\textquotedblleft degenerate versus nondegenerate\textquotedblright .

In the latter case the two-phase subroutine in Proposition \ref{IF2tp} can
be applied indefinitely to reach pointwise $C^{2,\gamma ^{\ast }}$
regularity for some universal $\gamma ^{\ast }$.

When $u$ falls into the degenerate case a new kind of dichotomy appears.
This is the deepest part of the paper. First of all, to run the subroutine
in Proposition \ref{IF} one needs to make sure that the closeness of $u^{-}$
to a purely quadratic profile makes $u^{+}$ to be a (viscosity) solution of
a one-phase free boundary problem with $u_{\nu }^{+}$ close to an
appropriate $\alpha$ on $F\left( u\right).$ This is the content of Lemma \ref%
{fbcondition}, in Section \ref{dicoD}. At this point two alternatives occur
at a smaller scale (Proposition \ref{first_step}).

\begin{itemize}
\item[D1] : either $u^{-}$ is closer to a purely quadratic profile at a
proper $C^{2,\gamma }$ rate and $u^{+}$ enjoys a $C^{2,\gamma}$ flatness
improvement;

\item[D2] : or $u^{-}$ is closer (at a $C^{2,\gamma }$ rate) to a one-phase
polynomial profile with a small non-zero slope but $u^{+}$ only enjoys an
\textquotedblleft intermediate\textquotedblright\ $C^{2}$ flatness
improvement.
\end{itemize}


If D1 occurs indefinitely we are done. If not, we prove that (Proposition %
\ref{it}) the intermediate improvement in D2 is kept for a while, at smaller
and smaller scale. The final and crucial step is to prove (Proposition \ref%
{together}) that, at a given universally small enough scale, the $%
C^{2,\gamma }$ one-phase approximation of $u^{-},$ together with the
intermediate $C^{2}$ flatness improvement of $u^{+},$ is good enough to
recover a full $C^{2,\gamma ^{\ast }}$ two-phase improvement of $u$ with a
universal $\gamma ^{\ast }<\gamma $.

As we have mentioned at the end of Section \ref{introduzione}, we emphasize that it is the
interplay between the parallel improvements on both sides of the free
boundary that makes possible to obtain the full two-phase improvement, at
the price of a little decrease of the H\"{o}lder exponent. This kind of
situation has no counterpart in the \emph{flatness implies }$C^{1,\gamma}$%
\emph{\ case} of \cite{DFS}.

From this point on we can go back to subroutine of Proposition \ref{IF2tp}
and finally reach pointwise $C^{2,\gamma ^{\ast }\text{ }}$%
regularity.\bigskip

In the next section we start implementing the above strategy. In the course
of a proof, universal constants possibly changing from line to line will be
denoted by $c,C.$ Dependence on other parameters, will be explicitly noted.

\section{Initial configurations}\label{initialconfig}

As we mentioned in Section \ref{strategy}, we start by showing that the flatness
condition \eqref{flat} allows us to normalize our solution so that a
rescaling $u_{\bar{r}}$ of $u$ satisfies a suitable $%
(V,\varepsilon ,\delta )$ flatness. We first recall the following result
proved in \cite{DFS}. Set 
\begin{equation*}
u_{r}(x):=\frac{u(rx)}{r},\quad {f_{\pm}}_r(x)=rf_{\pm }(rx),\quad x\in
B_{1}.
\end{equation*}%

\begin{lem}\label{normalize} Let $u$ be a (Lipschitz) solution to \eqref{fbtp} in $B_1$ with  $Lip(u) \leq L$ and $\|f_\pm\|_{L^\infty} \leq L$. For any $\eps >0$, $\eps< \eps_0=\eps_0(n,L)$,  there exist $\bar \eta$ depending on $\eps, n$ and $L$, $\bar \eta \leq \eps^4$, such that if $u$ satisfies \eqref{flat} for some $\eta \leq \bar \eta$ then
\be\label{conclusion_beta}\|u_r - U_{\beta}\|_{L^{\infty}(B_{1})} \leq \eps , \quad \text{for some $0 \leq \beta \leq L,$}\ee 
\be\|{f_\pm}_{r}\|_\infty \leq \eps, \quad |{f_\pm}_{r}(x)-{f_\pm}_{r}(0)| \leq \eps |x|^\gamma \ee
\be\{x_n \leq - \eps\} \subset \{u^+_{r}=0\} \subset \{x_n \leq \eps\}, \ee
and $r = \eps^3.$

\end{lem}
Let $u$ be as in Lemma \ref{normalize} and for a given $\eps$, let $\bar \eta(\eps)$ and $r(\eps)$ be the corresponding parameters provided by the lemma.

In the next Lemma, we denote by $\bar \delta, \bar \lambda$ the universal constants which will be chosen later in Proposition \ref{IF2tp} and Proposition \ref{first_step} (say for $\beta_1=L+1$).

\begin{lem}\label{first_step_1}There exists $\bar \eps$ universal such that if $u$ satisfies \eqref{flat} with $\bar \eta= \bar \eta(\bar \eps)$
then either of these flatness conditions holds with $\bar r=\bar r(\bar \eps).$
\begin{enumerate}
\item Degenerate case: 

$u^+_{\bar r}$ is $(V, \bar \lambda^{2+\gamma}, \bar \delta )$ flat in $B_1$, for 
$V=V^{1}_{0,e_n,a_n} \in \mathcal V_{f_+}$,  
$$|u^-_{\bar r}+ \frac 1 2{f_-}_{\bar r}(0) x_n^2| \leq \bar \delta^{1/2} \bar \lambda^{2+\gamma} \quad \text{in $B_1^-(u_{\bar r})$}$$
and
$$\|{f_-}_{\bar r}\|_\infty \leq \bar \delta, \quad |{f_\pm}_{\bar r}(x)-{f_\pm}_{\bar r}(0)| \leq \bar \delta |x|^\gamma$$

\item Non-degenerate case:

$u_{\bar r}$ is $(V, \bar \lambda^{2+\gamma}, \bar \delta)$ in $B_1$, 
with $V=V^{\alpha, \beta}_{0,e_n,a,b} \in \mathcal V_{{f_\pm}_{\bar r}, G},$
$$a'=b'=0, \quad \beta \geq \frac 1 2 \bar \delta^{1/2}\bar \lambda^{2+\gamma},$$
and
$$|{f_+}_{\bar r}(x)-{f_+}_{\bar r}(0)| \leq \bar \delta |x|^\gamma\quad |{f_-}_{\bar r}(x) - {f_-}_{\bar r}(0)| \leq \beta \bar \delta |x|^\gamma.$$
\end{enumerate}
\end{lem}

\begin{proof} 
Call $\eps^* = \bar \lambda^{2+\gamma}$ and $\tilde \delta = \frac 1 2 \bar \delta^{1/2} \eps^*$. Let $\bar \eps << \tilde \delta < \eps^*$ and $\bar \eps < \bar \delta$, to be made precise later. For such $\bar \eps$ the conclusion of Lemma \ref{normalize} above gives that in $B_1$
\be\label{+}\|u_{\bar r} - U_\beta\|_\infty \leq \bar \eps, \quad \|{f_\pm}_{\bar r }\|_\infty \leq \bar \eps, \quad |{f_\pm}_{\bar r}(x) - {f_\pm}_{\bar r}(0) | \leq \bar \eps |x|^\gamma \ee 
and 
\be\label{-} \{x_n \leq - \bar \eps\} \subset \{u^+_{\bar r}=0\} \subset \{x_n \leq \bar \eps\},\ee
for $\bar r=\bar r(\bar\eps)$ and some $0 \leq \beta \leq L.$

We distinguish two cases. For notational simplicity we drop the subindex $\bar r.$

\medskip

{\it Degenerate Case.} $\beta < \tilde \delta$. In this case we wish to show that
\be V(x_n -\eps^*) \leq u^+(x) \leq V(x_n +\eps^*)\ee
with
$$V(x_n):= (1+a_nx_n)x_n^+, \quad a_n=\frac{f_+(0)}{2}.$$
We prove the upper bound, as the lower bound can be obtained similarly. From the first two equations in \eqref{+} we get (for $\bar \eps << \eps^*$, $\alpha=G(\beta)$)
$$u^+ \leq \alpha x_n^+ + \bar\eps \leq \alpha V(x_n) + c \bar \eps \leq \alpha V(x_n + 2 \eps^* ),$$
hence
$$u^+ \leq V(x_n + 2\eps^*) + 2 \tilde \delta \leq V(x_n+\eps^*)$$
as long as $\tilde \delta \leq \frac 1 4 \eps^*$ (which is clearly satisfied because $\bar \delta$ is very small).

The bound for $u^-$ follows immediately by noticing that in view of \eqref{+} and \eqref{-},
$$|u^-+ \frac 1 2f_-(0) x_n^2| \leq \bar\eps + \alpha \bar \eps + \tilde \delta \quad \text{in $B_1^-(u)$.}$$

\medskip

{\it Non-degenerate Case.} $\beta \geq \tilde \delta.$ In this case we want to show that 
$$V(x_n -\eps^*) \leq u(x) \leq V(x_n+ \eps^*) \quad \text{in $B_1$}$$
with
$$V(x_n)=\alpha(1+a_nx_n)x_n^+-\beta(1+b_n x_n)x_n^-, \quad 2\alpha a_n= f_+(0),  \quad 2\beta b_n = f_-(0) \quad \alpha=G(\beta).$$
Let us prove the upper bound. In view of \eqref{+} we get,
$$u \leq V(x_n) + 2\bar\eps \leq v(x_n + \eps^*)$$
where in the last inequality we have used that $V' \geq \frac 1 2 \tilde \delta$ and $\bar \eps << \tilde \delta < \eps^*.$
The bound on the modulus of continuity of $f_-$ also follows because $\beta \geq \tilde \delta >> \bar \eps.$
\end{proof}

We conclude this section by providing sufficient conditions for a two-phase/one-phase polynomial $V$ to be a strict subsolution (resp. supersolution). We will work simultaneously with the two-phase problem \eqref{fbtp} 
and with the one-phase problem
\be\label{fb_almost_2}\begin{cases}
\Delta v = f_+ \quad \text{in $B_1^+(v),$}\\
|v^+_\nu - \alpha|\leq \delta^{1/2}\eps \quad \text{on $F(v),$}\
\end{cases}
\ee
with $\delta$ and $\eps$ sufficiently small constants. The free boundary will always be $C^{1,\gamma}$.

The two-phase results are needed to deal with the non-degenerate case, that is the case when our flat solution $u$ to \eqref{fbtp} is trapped between two translates of a function $V \in \mathcal V_{f_\pm,G}$ with a positive slope $\beta$ (not too small). 
The one-phase results will be of use  when we will deal with the degenerate case, that is when the flatness of the free boundary only guarantees closeness of the positive part $u^+$ to a quadratic profile, i.e. $\beta=0.$ 

Precisely we prove the following Proposition. The corresponding statement for $V$ to be a strict supersolution can also be obtained. Here, $0 \leq \beta \leq \beta_1, 1=G(0) \leq \alpha \leq \alpha_1=G(\beta_1)$. Dependence of the constants on $\beta_1$ is not noted (as it will be fixed universal.)

\begin{prop}\label{sub} Assume that in $B_1$
\be\label{f+}|f_+(x)-f_+(0)| \leq \delta \eps, \quad \text{and} \quad |f_-(x)-f_-(0)| \leq \beta \delta \eps \quad \text{if $\beta \neq 0$.}\ee Given $V=V^{\alpha,\beta}_{M,\omega, a, b}$ with
\be \label{Mabsmall}\|M\|, |a|, |b'| \leq \delta \eps^{1/2}, \ee 
\be\label{new}|b_n|\leq \bar C \delta^2, \quad |b_n|\|M\|  \leq \bar C \delta^2\eps,\ee
$V$ is a strict subsolution to 
\eqref{fbtp} for $\beta \neq 0$ or to \eqref{fb_almost_2} for $\beta=0$,
if 
\be\label{1+}2 \alpha a \cdot \omega - \alpha tr M \geq f_+(0) +  2\delta \eps\ee
\be \label{2+}2 \beta b \cdot \omega  - \beta tr M \geq f_-(0) +  2\beta \delta \eps \quad \text{if $\beta \neq 0$}\ee
and 
\be\label{4+} \alpha+ \alpha t a \cdot \omega^\perp \geq G(\beta) + \beta G'(\beta) t b \cdot \omega^\perp + \delta^{1/2}\eps, \quad \forall t \in [0,1]\ee
as long as $\delta$ is small (depending on $\bar C$).
\end{prop}
\begin{proof} Say $\omega = e_n.$ Since $|a|, |b| <1$,  we have that
$$V(x)= \alpha(x_n - \frac 12 x^T M x)(1+a\cdot x), \quad \text{in $B_1^+(V)$}.$$
Thus,
$$\Delta V = -\alpha tr M (1+a \cdot x) +2\alpha (e_n - \frac 1 2 M x)\cdot a \quad \text{in $B_1^+(V)$}.$$
Using assumptions \eqref{f+}-\eqref{Mabsmall}-\eqref{new}-\eqref{1+} we get
$$\Delta V \geq f_+(0)+ \delta \eps - C \delta^2 \eps \geq f_+(x) + \delta \eps  - C \delta^2 \eps > f_+(x)$$ if $\delta$ is chosen universally small.
The computation in the negative phase follows similarly. 

To check the free boundary condition we must verify that, say for $\beta>0,$
\be \label{fbsub}|\nabla V^+| - G(|\nabla V^-|)> 0 \quad \text{on $F(V)$.}\ee
We compute that on $F(V)$, since $M \cdot e_n=0,$
$$|\nabla V^+|=\alpha |e_n - M x | (1+a \cdot x) \geq \alpha (1+ a\cdot x).$$
Similarly, using assumption \eqref{Mabsmall} 
$$|\nabla V^-| = \beta|e_n - M x|(1+ b\cdot x)\leq \beta (1+b\cdot x) + C \delta^2 \eps.$$
Thus, 
$$G(|\nabla V^-|) \leq G(\beta)+ \beta G'(\beta) b \cdot x + C \delta^2 \eps$$
and \eqref{fbsub} is satisfied in view of \eqref{Mabsmall}-\eqref{new}-\eqref{4+}, as long as $\delta$ is small enough. In fact, \eqref{4+} gives that 
$$\alpha (1+ a\cdot x) \geq G(\beta)+ \beta G'(\beta) b \cdot x + \alpha a_n x_n - \beta G'(\beta) b_n x_n + \delta^{1/2}\eps. $$
Using that on $F(V)$ the size of $x_n$ is bounded by $\|M\|$ and from assumptions \eqref{Mabsmall}-\eqref{new}, we conclude that
$$\alpha (1+ a\cdot x) \geq G(\beta)+ \beta G'(\beta) b \cdot x - C \delta^2 \eps +\delta^{1/2}\eps,$$ from which the desired claim follows.

\smallskip

A similar computation holds for $\beta=0.$
\end{proof}

\begin{rem}\label{sub_rem} We can consider the larger class of functions
$$P:= x\cdot \omega + \xi' \cdot x' -  \frac 1 2 x^T M x,$$
and the corresponding $V$'s (for $A \in \R$)
$$V= \alpha(1 + a \cdot (x+Ae_n)) P^+ - \beta (1 + b \cdot (x+Ae_n))P^-.$$ The proposition above remains valid if $|\xi '|, |A|\leq \bar C \delta \eps^{1/2},$ $|b_n||A| \leq \bar C \delta^2 \eps.$\end{rem}



\section{The Improvement of flatness}\label{improvflat}

In this section we prove our main improvement of flatness theorem. 
Let $u$ solve
$$\Delta u^+= f_+ \quad \text{in $B_1^+(u)$}, \quad \Delta u^-= f_- \quad \text{in $B_1^-(u)$}, \quad \text{if $\beta>0$}$$ 
or
$$\Delta u^+= f_+ \quad \text{in $B_1^+(u)$}, \quad \text{if $\beta=0$}$$ 
with $0 \in F(u) \in C^{1,\gamma}$.

Denote by $V=V^{\alpha, \beta}_{M, e_n, a, b} \in \mathcal V_{f_\pm,G}$ with $0 \leq \beta \leq \beta_1, 1=G(0) \leq \alpha \leq \alpha_1:=G(\beta_1).$ In what follows, given $V$, for any function $v$ defined in $B_1$, we will use the notation:
\be\label{tilde_v}\tilde v^{\eps}(x)= \begin{cases}\dfrac{v(x) - \alpha(1+a\cdot x)P_{M,e_n}}{\alpha\eps}, \quad x \in
B_1^+(u) \cup F(u), \\ \ \\ \dfrac{v(x) - \beta (1+ b \cdot x)P_{M,e_n}}{\beta\eps}, \quad x \in
B_1^-(u), \quad \beta>0,\\
0, \quad \quad x \in
B_1^-(u), \quad \beta=0.\ \end{cases}
\ee

\begin{prop}[Improvement of Flatness]\label{IFtp}There exist $\bar \eta, \bar \delta, \bar \eps$ universal, such that 
\begin{enumerate}
\item Two-phase case: $\beta>0,$ if
\begin{equation}\label{flat1tp_e}
 \text{$u$ is $(V,\eps, \bar \delta)$ flat in $B_{1},$} \quad 0<\eps \leq \bar\eps\end{equation}
\begin{equation}\label{flat2tp_e}
|f_+(x) - f_+(0)| \leq \bar \delta \eps, \quad |f_-(x) - f_-(0)| \leq \beta \bar \delta \eps, 
\end{equation}
and
\be\label{fbtp_e} u_\nu^+ = G(u_\nu^-)  \quad \text{on $F(u) \cap B_{2/3}$}
\ee
then $$\text{ $u$ is $(\bar V, \bar\eta^{2+\gamma}\eps, \bar\delta)$ flat in $B_{\bar\eta}$}$$
with $\bar V=V^{\bar \alpha, \bar \beta}_{\bar M, \bar \nu, \bar a, \bar b} \in \mathcal V_{f_\pm,G}$ and $|\beta -\bar \beta| \leq C \eps$, for $C$ universal.

\item One-phase case: $\beta =0$, if
\begin{equation}\label{flat1_e} \text{$u^+$ is $(V,\eps, \bar\delta)$ flat in $B_1$}, \quad 0<\eps \leq \bar\eps,
\end{equation}
\begin{equation}\label{flat2_e}
|f_+(x) - f_+(0)| \leq \bar\delta \eps, 
\ee
and
\be\label{fb_e} |u_\nu^+ - \alpha| \leq \bar\delta^{1/2}  \eps \quad \text{on $F(u) \cap B_{2/3}$}
\ee
then 
\begin{equation}
\text{$u^+$ is $(\bar V, \bar\eta^{2+\gamma}\eps, \bar\delta)$ flat in $B_{\bar\eta}$}
\end{equation}
\end{enumerate}
with $\bar V=V^{\alpha}_{\bar M, \bar \nu, \bar a_{\bar \nu}} \in \mathcal V_{f_+}$.\end{prop}

\begin{proof} Let $\bar\eta$ be given (to be specified later).

{\bf Step 1.} By contradiction assume that there exist $\eps_k, \delta_k \to 0$ and $u_k, V_k, {f_\pm}_k,G_k$ as above, with $\|G_k\| \leq L, G_k(0)=1,$ for which the assumptions above hold but the conclusion does not. 
Now, let us define the corresponding $\tilde u_k^{\eps_k}$ as in \eqref{tilde_v}. For notational simplicity we call $w_k:=\tilde u_k^{\eps_k}$ 

Then \eqref{flat1tp_e}-\eqref{flat1_e} give,
\begin{equation}\label{flat_tilde**} -2 \leq w_{k}(x) \leq 2
\quad \text{for $x \in B_1$}.
\end{equation}
Up to  a subsequence, $G_k$ converges, locally uniformly in $C^1$, to some $C^1$ function $G_0$, while
$\beta_k\to \tilde \beta$ so that $\alpha_k\to \tilde \alpha
=G_0(\tilde{\beta}).$
Moreover, by Harnack inequality (see Lemma \ref{HI} in the next section) the graphs of $w_k$ converge in the Hausdorff distance to a H\"{o}lder continuous $w.$

\vspace{2mm}

\textbf{Step 2 -- Limiting Solution.} We now show that, say in the case $\beta_k>0$ for all $k$'s, $w$
solves the following linearized problem
\begin{equation}\label{Neumann}
  \begin{cases}
    \Delta w=0 & \text{in $B_{1/2} \cap \{x_n \neq 0\}$}, \\
\ \\
\tilde \alpha w_n^+ - \tilde\beta G'_0(\tilde \beta) w_n^-=0 & \text{on $B_{1/2} \cap \{x_n =0\}$}.
  \end{cases}\end{equation}

One can argue similarly in the case $\beta_k=0$ for all $k$'s, with $w$ satisfying:
\begin{equation}\label{Neumann2}
  \begin{cases}
    \Delta w=0 & \text{in $B_{1/2} \cap \{x_n > 0\}$}, \\
\ \\
w_n =0 & \text{on $B_{1/2} \cap \{x_n =0\}$}.
  \end{cases}\end{equation}

It is easy to check that, from our assumptions,  $$|\Delta  w_{k}| \leq C\delta_k \quad
\text{in $B_1^+(u_k) \cup B^-_1(u_k)$},$$ hence one easily deduces that $w$ is harmonic in $B_{1/2} \cap \{x_n  \neq 0\}$.

Next, we prove that $w$ satisfies the boundary
condition in \eqref{Neumann} in the viscosity sense.

Let $\phi$ be a function of the form
$$\phi(x) = A+ px_n^+- qx_n^- + \xi' \cdot x' +\frac 12  x^T N x$$ with $$N \in S^{n\times n},  tr N=0, \quad A \in \R,$$ and \be\label{FBcond}\tilde \alpha p- \tilde \beta G'_0(\tilde \beta) q>0.\ee

Then we must show that $\phi$ cannot touch $w$ strictly by below at a point $x_0= (x_0', 0) \in B_{1/2}$ (the analogous statement by above follows with a similar argument.)

Suppose that such a $\phi$ exists and let $x_0$ be the touching point. 
Without loss of generality, we can assume that $\phi$ is globally below $w$ (this observation will be used in the Remark \ref{normals_v}.)

We construct now a subsolution to the free boundary problem satisfied by the $u_k$'s. We will use these computations also in the proof of Lemma \ref{HI} in the next section.

Call $$W_k(x)= \tilde \alpha_k(1+\tilde c_k \cdot (x+\eps_k Ae_n))Q^+ - \tilde\beta_k(1+\tilde d_k \cdot (x+\eps_k Ae_n))Q^-,$$
where
$$Q=P_{\tilde M_k,e_n} + \eps_k \xi' \cdot x' + A\eps_k, \quad \tilde M_k =  M_k - \eps_k N,$$
$$\tilde \alpha_k= \alpha_k(1+\eps_k p), \quad \tilde\beta_k=\beta_k(1+\eps_k q)$$
$$\tilde c_k = a_k + \eps_k c, \quad \tilde d_k = b_k +\eps_k d $$
and $c,d$ to be specified later.

From Proposition \ref{sub} and Remark \ref{sub_rem}, $W_k$ is a strict subsolution for $k$ large as long as

$$2 \tilde\alpha_k \tilde c_k \cdot e_n - \tilde \alpha_k tr  M_k \geq {f_{+}}_{k}(0) + 2\delta_k \eps_k$$
$$2 \tilde \beta_k \tilde d_k \cdot e_n - \tilde\beta_k tr M_k \geq {f_-}_{k}(0) + 2\tilde\beta_k\delta_k \eps_k $$
and 
$$ \tilde\alpha_k+ \tilde\alpha_k  \tilde c_k \cdot x' \geq G_k(\tilde\beta_k) + \tilde\beta_k G'_k(\tilde\beta_k) \tilde d_k \cdot x' + \delta_k^{1/2} \eps_k.$$

The first two equations are satisfied if we choose $$c_n = C(p) \delta_k, \quad d_n=C(q)\delta_k.$$
Indeed, say for the first one, we compute,
$$2 \tilde\alpha_k \tilde c_k \cdot e_n - \tilde\alpha_k tr M_k=  {f_+}_k(0) + 2\alpha_k \eps_k c_n + 2\alpha_k \eps_k p a_k \cdot e_n + 2\alpha_k \eps_k^2 p c_n - \eps_k \alpha_k p tr M_k $$ 
$$\geq {f_+}_k(0) + 2\alpha_0 \eps_k c_n + O(p\eps_k^{3/2} \delta_k),$$
and the conclusion follows with an appropriate choice of $C.$ 

We will also choose $c'=d'=0.$ Then, the third equation is satisfied for $k$ large in view of \eqref{FBcond}.

In fact,
$$G_k(\tilde\beta_k) + \tilde\beta_k G'_k(\tilde\beta_k) \tilde d_k \cdot x' \leq G_k(\beta_k) + \beta_k G'_k(\beta_k)q\eps_k+ \beta_k G'_k(\beta_k)b_k \cdot x' + \eps_kO(|q|(\eps_k+ \delta_k))$$
$$= \alpha_k + \alpha_k a_k\cdot x' + \beta_k G'_k(\beta_k)q \eps_k+  \eps_kO(|q|(\eps_k+ \delta_k)).$$
Thus we need,
$$\alpha_k p - \beta_k G'_k(\beta_k)q \geq C(p,q) \delta_k + \delta_k^{1/2}$$
which is satisfied for $k$ large in view of \eqref{FBcond}.

Define now, as in \eqref{tilde_v}, $W^*_k:= \tilde W_k^{\eps_k}$. We observe that $W^*_k$ converges uniformly to $\phi$ on $B_{1/2}.$ 
Indeed, one can easily compute that in $B_1^+(W_k)$
$$W_k(x) - \alpha_k (1+a_k \cdot x)P_{M_k,e_n}= \alpha_k \eps_k(p x_n +\frac 1 2 x^T N x + \xi' \cdot x' + A) + \alpha_k \eps_k O(\delta_k)$$
and similarly, in $B_1^-(W_k)$
$$W_k(x) - \beta_k (1+b_k \cdot x)P_{M_k,e_n}= \beta_k \eps_k(q x_n + \frac 1 2 x^T N x + \xi' \cdot x' + A) + \beta_k \eps_k O(\delta_k).$$

Since $w_k$ converges uniformly to $w$, $W^*_k$ converges uniformly to $\phi$ and $\phi$ touches $w$ strictly by below at $x_0$ we can conclude that there exist a sequence of constants $t_k \to 0$ and of points $x_k \to x_0$ such that the function
$$\psi_k(x)=W_k(x+\eps_k t_k e_n)$$ touches $u_k$ by below at $x_k$. We thus get a contradiction if we prove $\psi_k$ is a strict subsolution to the free boundary problem satisfied by the $u_k$'s. This follows from the fact that $W_k$ is a strict subsolution and the translation in the $e_n$direction only perturbs $A$ into $A+ t_k$.

\medskip

{\bf Step 3.} Since $w_k$ converges uniformly to $w$ and $w(0)=0$ we get that
$$|w_k - \psi(x)| \leq \frac 1 8 \bar\eta^{2+\gamma}, \quad \text{in $B_{\bar\eta}$}$$
with
$$\psi(x) = p x_n^+ - qx_n^- - \xi' \cdot x' - \frac 1 2  x^T Nx + \hat a \cdot x x_n^+ - \hat b \cdot x x_n^- $$
and
$$\tilde \alpha p - \tilde \beta G'_0(\tilde \beta)q=0, \quad tr N=2\hat a_n=2\hat b_n, \quad N e_n=0$$
$$\tilde \alpha \hat a' \cdot x' = \tilde \beta G'_0(\tilde \beta) \hat b' \cdot x'.$$
Thus,
$$|u_k - \alpha_k(1+a_k \cdot x)P_{M_k,e_n} - \alpha_k \eps_k\psi(x)| \leq \frac 1 8 \alpha_k \eps_k \bar\eta^{2+\gamma}, \quad \text{in $B_{\bar\eta} \cap (B^+_1(u_k) \cup F(u_k)):= B^+_F$}$$
and
$$|u_k - \beta_k(1+b_k \cdot x)P_{M_k,e_n} - \beta_k \eps_k\psi(x)| \leq \frac 1 8 \beta_k \eps_k \bar\eta^{2+\gamma}, \quad \text{in $B^-_{\bar\eta}(u_k).$}$$
Hence, since in the region $B^+_F$ we have $x_n \geq -2 \eps$, we conclude that
$$|u_k - \alpha_k(1+a_k \cdot x)P_{M_k,e_n} - \alpha_k \eps_k(p x_n- \xi' \cdot x'  -\frac 1 2 x^TN x + \hat a\cdot x x_n)| \leq \frac 1 4 \alpha_k \eps_k \bar\eta^{2+\gamma}, \quad \text{in $B^+_F$}$$
and similarly
$$|u_k - \beta_k(1+b_k \cdot x)P_{M_k,e_n} - \beta_k \eps_k(q x_n- \xi' \cdot x'  - \frac 1 2 x^TN x + \hat b\cdot x x_n)| \leq \frac 1 4 \beta_k \eps_k \bar\eta^{2+\gamma}, \quad \text{in $B^-_{\bar\eta}(u_k).$}$$

Set
$$\alpha^*_k=\alpha_k(1-\eps_k p), \quad \beta_k^*=\beta_k(1-\eps_k q),$$
$$N^*_k = M_k + \eps_k N, \quad \nu^*_k= e_n + \eps_k \xi, \quad \xi_n=0,$$
and 
$$a^*_k= a_k + \eps_k\hat a, $$
$$b^*_k= b_k + \eps_k\hat b.$$

Then,
$$|u_k - \alpha^*_k(1+a^*_k \cdot x)P_{N^*_k,\nu^*_k}| \leq \frac 1 2 \alpha^*_k \eps_k \bar\eta^{2+\gamma}, \quad \text{in $B^+_F$}$$
and
$$|u_k - \beta^*_k(1+b^*_k \cdot x)P_{N^*_k,\nu^*_k}| \leq \frac 1 2 \beta^*_k \eps_k \bar\eta^{2+\gamma}, \quad \text{in $B^-_{\bar\eta}(u_k)$.}$$

Now choose,
$$\bar \nu_k = \frac{\nu^*_k}{|\nu^*_k|} = e_n + \eps_k \xi + \eps_k^2 \tau, \quad |\tau|\leq C.$$
For notational simplicity, we drop the dependence on $k$ from $\nu^*_k$ and $\bar \nu_k.$

We write 
$$N^*_k= \bar N_k+L_k, \quad \bar N_k \cdot \bar \nu =0.$$
Then, since $M_k  e_n =N e_n=0,$ we get
$$\|L_k\| = O(\eps_k^{3/2}).$$
Moreover set,
$$\bar \beta_k= \beta^*_k \quad \bar \alpha_k= G_k(\bar \beta_k)= \alpha^*_k + O(\eps_k^2)$$
where we have used that $\tilde \alpha p - \tilde \beta G'_0(\tilde \beta)q=0$.

Thus,
$$|u_k - \bar \alpha_k(1+a^*_k \cdot x)P_{\bar N_k,\bar \nu}| \leq \frac 2 3 \bar \alpha_k \eps_k \bar\eta^{2+\gamma}, \quad \text{in $B^+_F$}$$
and
$$|u_k - \bar \beta_k(1+b^*_k \cdot x)P_{\bar N_k,\bar \nu}| \leq \frac 2 3 \bar \beta_k \eps_k \bar\eta^{2+\gamma}, \quad \text{in $B^-_{\bar\eta}(u_k)$.}$$

Let $S_{\bar \nu}:=\{\bar \nu_i\}_{i=1,\ldots n}$ be an orthonormal system containing $\bar \nu.$ Say, $\bar \nu=\bar \nu_n.$
Recall that
$$\bar \nu_n = e_n + \eps \omega, \quad |\omega|\leq C,$$
hence
$$\bar \nu_i = e_i + \eps v, \quad  |v| \leq C, \quad i \neq n$$
Define,
$$\bar a_k= a^*_k + \sum_{i=1}^{n}z_i \bar \nu_i $$
$$\bar b_k= b^*_k + \zeta \bar \nu_n $$
where $z_i, \zeta$ are chosen so that
$$2\bar \alpha_k \bar a_k \cdot \bar \nu_n - \bar \alpha_k tr \bar N_k = {f_+}_k(0)$$
$$2\bar \beta_k \bar b_k \cdot \bar \nu_n - \bar \beta_k tr \bar N_k = {f_-}_k(0), \quad \text{if $\beta_k >0$}$$
and 
$$\bar \alpha_k \bar a_k \cdot \bar \nu_i= \bar \beta_k G'_k(\bar \beta_k) \bar b_k \cdot \bar \nu_i, \quad i \neq n.$$
Unravelling these identities using all of the definitions above and the compatibility conditions for $\psi$ we can estimate that
$$|z_i|, |\zeta| = O(\eps_k^{3/2}).$$

Therefore we conclude that
$$|u_k - \bar \alpha_k(1+\bar a_k \cdot x)P_{\bar N_k,\bar \nu}| \leq \frac 4 5 \bar \alpha_k \eps_k \bar\eta^{2+\gamma}, \quad \text{in $B^+_F$}$$
and
$$|u_k - \bar \beta_k(1+\bar b_k \cdot x)P_{\bar N_k,\bar \nu}| \leq \frac 4 5 \bar \beta_k \eps_k \bar\eta^{2+\gamma}, \quad \text{in $B^-_{\bar\eta}(u_k)$.}$$

Set, $\bar V_k= V^{\bar \alpha_k, \bar \beta_k}_{\bar N_k, \bar \nu, \bar a_k, \bar b_k}$ we conclude that 
$$\bar V(x- \eps_k \bar \nu) \leq u_k(x) \leq \bar V(x + \eps_k \bar \nu) \quad \text{in $B_{\bar\eta}$}$$
and we reach a contradiction as long as 
$$\bar \eta |\bar a_k|, \bar \eta|\bar b_k \cdot \bar \nu^\perp|, \bar \eta\|\bar M_k\| \leq \delta_k (\eps_k \bar \eta^{1+\gamma})^{1/2}$$
and 
$$\bar \eta|\bar b_k \cdot \bar \nu| \leq \delta_k^2, \quad \bar \eta^2|\bar b_k \cdot \bar \nu| \|\bar M_k\| \leq \delta_k^2 (\eps_k \bar \eta^{1+\gamma}).$$

This follows (for $\bar \eta$ possibly smaller) from the initial bounds on $|a_k|, |b_k|, |M_k|$ and the fact that 
$$|a_k-\bar a_k|, |b_k-\bar b_k|, \|M_k-\bar M_k\|, |e_n - \bar \nu| \leq C \eps_k.$$

\end{proof}

\begin{rem}\label{normals_v} We observe that  it is enough for the free boundary condition to be satisfied in the following viscosity sense. At all points where $u$ is touched globally by the positive side in $B_{2/3}$ by a test function with free boundary $\Gamma:= \{x_n= \mathcal P\}$
with $\mathcal P$ a  quadratic polynomial with coefficients of size $1$, then
$$u_\nu^+  \leq G(u_\nu^-) \quad \text{if $\beta>0$}, \quad u_\nu^+ - \alpha \leq \bar\delta^{1/2} \eps \quad \text{if $\beta=0$}$$
and similarly the lower bound is satisfied at all points where $u$ is touched by above on the free boundary by a surface $\Gamma$ as before.

This can be easily seen from the proof (see the conclusion of Step 2 in the proof above.)
\end{rem}

As a consequence we obtain the following two propositions. Let $u$ be as at the beginning of this section and $0 \leq \beta \leq \beta_1, 1 \leq \alpha \leq \alpha_1.$

\begin{prop}[One-phase $C^{2,\gamma}$ improvement of flatness]\label{IF} There exist $\bar\eta, \bar\delta, \bar\lambda$ such that if for $\beta=0$
\begin{equation}\label{flat1}
 \text{$u^+$ is $(V, \lambda^{2+\gamma}, \bar\delta)$ flat in $B_{\lambda}, \lambda \leq \bar\lambda$}
\end{equation} with $V=V^{\alpha}_{M, e_n, a_n} \in \mathcal V_{f_+},$
\begin{equation}\label{flat2}
|f_+(x) - f_+(0)| \leq \bar\delta |x|^\gamma
\end{equation}
and
\be\label{1phase} |u^+_\nu -\alpha| \leq   \bar\delta^{1/2} \lambda^{1+\gamma} \quad \text{on $F(u) \cap B_{2/3 \lambda},$}\ee in the viscosity sense,
then \begin{equation}\label{flat3}
\text{$u^+$ is $(\bar V, (\bar\eta\lambda)^{2+\gamma}, \bar\delta)$ in $B_{\bar\eta\lambda}$}
\end{equation}
with $\bar V= V^\alpha_{\bar M, \bar {\nu}, \bar a_{\bar{\nu}}} \in \mathcal V_{f_+}$.
\end{prop}

\begin{prop}[Two-phase $C^{2,\gamma}$ improvement of flatness]\label{IF2tp} There exist $\bar\eta, \bar\delta, \bar\lambda$ universal, such that if for $\beta>0$
\begin{equation}\label{flat1tp_l}
 \text{$u$ is $(V, \lambda^{2+\gamma}, \bar\delta)$ flat in $B_{\lambda}, \lambda \leq \bar\lambda$}\end{equation}with  $V=V^{\alpha, \beta}_{M, e_n, a, b} \in \mathcal V_{f_\pm, G},$
\begin{equation}\label{flat2tp_l} 
|f_+(x) - f_+(0)| \leq \bar\delta |x|^\gamma, \quad |f_-(x) - f_-(0)| \leq \beta \bar\delta |x|^\gamma
\end{equation}
and
$$u_\nu^+=G(u_\nu^-) \quad \text{on $F(u) \cap B_{2/3\lambda}$}$$
then 
\begin{equation}\label{flat3tp_l}
\text{$u$ is $(\bar V, (\bar\eta\lambda)^{2+\gamma}, \bar\delta)$ in $B_{\bar\eta\lambda}$}
\end{equation}
with $\bar V=V^{\bar \alpha, \bar \beta}_{\bar M, \bar \nu, \bar a, \bar b} \in \mathcal V_{f_\pm,G}$ and $|\beta -\bar \beta| \leq C \lambda^{1+\gamma}$ for $C$ universal.
\end{prop}

\section{Harnack inequality}\label{hanaineq}

In this section we prove a Harnack type inequality which is the key ingredient in the compactness argument used to prove our improvement of flatness proposition.

Let $V=V^{\alpha, \beta}_{M, e_n, a, b} \in \mathcal V_{f_\pm,G}$ with $0 \leq \beta \leq \beta_1, 1 \leq \alpha \leq \alpha_1.$ Let $u$ solve
$$\Delta u^+= f_+ \quad \text{in $B_1^+(u)$}, \quad \Delta u^-= f_- \quad \text{in $B_1^-(u)$}, \quad \text{if $\beta>0$}$$ 
or
$$\Delta u^+= f_+ \quad \text{in $B_1^+(u)$}, \quad \text{if $\beta=0$}$$ 
with $0 \in F(u) \in C^{1,\gamma}$. 

We need the following key lemma. The free boundary condition in the lemma is assumed to hold in the viscosity sense of Remark \ref{normals_v}.

\begin{lem}\label{HI} There exist $\bar \eps, \bar \delta$  universal such that 
\begin{enumerate}
\item Two-phase case: $\beta >0$, 
if 
\begin{equation}\label{flat1tp_h}
\text{$u$ is $(V,\eps, \bar \delta)$ flat in $B_1$}, \quad 0<\eps \leq \bar\eps, \end{equation}
\begin{equation}\label{flat2tp_h}
|f_+(x) - f_+(0)| \leq \bar\delta \eps, \quad |f_-(x) - f_-(0)| \leq \beta \bar\delta \eps,
\end{equation}
and
\be\label{fbtp_h} u_\nu^+ = G(u_\nu^-)  \quad \text{on $F(u) \cap B_{2/3}$}
\ee
then either
$$u(x) \leq V(x+ (1-\eta)\eps e_n) \quad \text{in $B_\eta$}$$
or
$$u(x) \geq V(x - (1-\eta)\eps e_n) \quad \text{in $B_\eta$}$$
for a small universal $\eta \in (0,1).$\\
\item One-phase case: $\beta =0$, if
\begin{equation}\label{flat1_h}
\text{$u^+$ is $(V, \eps, \bar \delta)$ flat in $B_{1},$} \quad 0<\eps \leq \bar \eps,
\end{equation}
\begin{equation}\label{flat2_h}
|f_+(x) - f_+(0)| \leq \bar\delta \eps, 
\ee
and
\be\label{fb_h} |u_\nu^+ - \alpha| \leq \bar\delta^{1/2}  \eps \quad \text{on $F(u) \cap B_{2/3}$}
\ee
\end{enumerate}
then either
$$u^+(x) \leq V(x+ (1-\eta)\eps e_n) \quad \text{in $B_\eta$}$$
or
$$u^+(x) \geq V(x - (1-\eta)\eps e_n) \quad \text{in $B_\eta$}$$
for a small universal $\eta \in (0,1).$
\end{lem}
\begin{proof} To fix ideas let $\beta>0$. The case $\beta=0$ follows in the same way.

Let $\bar x =\frac 1 5 e_n$ and assume that
\be\label{u-p>ep2tp} u(\bar x) \geq V(\bar x)>0.\ee
We prove that the second statement holds.

Define:
\be \label{phi}
\phi_t(x):= t+ p x_n^+ - 2n x_n^- + \frac 1 2 x^T N x
\ee
where $N \in \mathcal S^{n\times n}$
$$N_{ii}=- 2, \quad \text{if $i, j=1, \ldots n-1$}$$
$$N_{1,j}=0 \quad \text{if $j=1,\ldots n-1$}$$
$$N_{nn}=4n.$$
and
\be\alpha p =1+ 2n\beta G'(\beta). \ee

We show that there  is a constant $ r_0 \leq 1/16$ ($r_0$ universal) such that 
\be\label{negative} \phi_{1/8} <-1/16 \quad \text{on $-1/2 \leq x_n \leq r_0, |x'| =1/2$ and on $x_n=-1/2, |x'| \leq 1/2$.}\ee Indeed, on the first region above, when $x_n > 0$ we have:
$$\phi_{1/8}(x)= 1/8+ 2n x_n^2 - 1/4 + p x_n \leq -1/8+ (2n + p) r_0  < -1/16$$
as long as
$$(2n+p) r_0 \leq 1/16.$$
When $-1/2 \leq x_n \leq 0$ and $|x'|=1/2$
$$\phi_{1/8}(x)= 1/8+ 2n x_n(x_n+1) - 1/4  < -1/16.$$
Finally, on $x_n=-1/2$ and $|x'|\leq 1/2$ we have
$$\phi_{1/8}(x)= 1/8 - |x'|^2  - n/2 < -1/16.$$

Finally notice that ($\eta$ universal ),
\be \label{big} \phi_{1/8} \geq \frac{1}{16} \quad \text{in $B_\eta$}.\ee

Call
$$D'':=\{|x'| \leq 2/3, r_0/2 \leq x_n \leq 2/5\}$$
and
$$D' :=\{|x'| \leq 1/2, r_0/3 \leq x_n \leq 1/5\}.$$
Finally,
$$D :=\{|x'| \leq 1/2, -1/2 \leq x_n \leq r_0\}.$$
Notice that (for $\eps$ small)
\be\label{inclusion1} D''\subset B_1^+(V^\eps) \subset B_1^+(u). \ee 

We have that  $v(x):=u(x) -V(x-\eps e_n) \geq
0$ and because of \eqref{flat2tp_h} and the sizes of the coefficients of $V$, we have $\Delta v \geq -\bar\delta \eps - \bar c \alpha \bar \delta^2 \eps$ in $D''$. Thus we can apply Harnack inequality
to obtain
$$ v \geq cv(\bar x) - C(\bar\delta \eps + \bar c \alpha \bar \delta^2\eps) \quad \text{in $D'$}. $$
From \eqref{u-p>ep2tp} we conclude that (for $\bar\delta$ small enough)
\be\label{u-p>cep} u(x)
- V(x-\eps e_n)\geq c\alpha \eps - C \bar\delta\eps  - \bar c \alpha \bar\delta^2\eps \geq \alpha c_0\eps \quad \text{in $D'$}. \ee

Now set,
$$\tilde \alpha= \alpha(1+\eps c_0 p), \quad \tilde \beta=\beta(1+\eps 2 n c_0 ), \quad \tilde M= M - \eps c_0N, \quad \tilde c= a+\eps c, \quad \tilde d= b+ \eps d$$
with 
$$c_i= d_i =0 \quad \forall i=1,\ldots n-1, \quad c_n=d_n=O(\bar \delta),$$
and call 
$$W_t :=\tilde \alpha(1+ \tilde c \cdot (x-\eps e_n+ t\eps e_n))(P_{\tilde M,e_n} -\eps+ t\eps)^+- \tilde\beta(1+ \tilde d \cdot (x- \eps e_n+t\eps e_n))(P_{\tilde M,e_n}-\eps + t\eps)^-.$$
Now, the same computation as in Step 2 of Proposition \ref{IFtp} guarantees that $\tilde W^\eps_t$ (defined as in \eqref{tilde_v}) converges uniformly to $-1+c_0 \phi_t$ as $\eps \to 0$. Clearly, $\tilde V^\eps(x-\eps e_n)= -1 + O(\bar \delta^{1/2}\eps)$. Since for $\bar t <<0$ we have $c_0 \phi_{\bar t} \leq -2$ in $D$, we conclude that 
$$ W_{\bar t} \leq V(x-\eps e_n) \leq u \quad \text{on $D$.}$$
Let $\bar s $ be the largest $t$ such that 
$$ W_t \leq u \quad \text{on $D$.}$$
We want to show that $ \bar s \geq 1/8.$ Then, in view of the uniform convergence of  $\tilde W^\eps_{\bar s}$ to $c_0 \phi_{\bar s}$ and the bound \eqref{big} we get 
$$u(x) \geq W_{\bar s} \geq V(x- (1-\bar c)\eps e_n) \quad \text{in $B_{\eta}.$}$$

Assume that $\bar s< 1/8.$ Then the first touching point $\tilde x$ of $u$ and $W_{\bar s}$ in $D$, occurs on $x_n=r_0.$ Indeed, as shown in Step 2 of Proposition \ref{IFtp}, $W_{\bar s}$ is a strict subsolution to \eqref{fbtp} which lies below $u$ in $D$. Thus the touching point can only occur on $\p D.$ 

By \eqref{negative} and the uniform convergence of the  $\tilde W^\eps_{\bar s}$ we have $W_{\bar s} < V^\eps \leq  u$ on $\p D \setminus \{x_n=r_0.\}$ Thus, again by the uniform convergence of the  $\tilde W^\eps_{\bar s}$ and the fact that $\phi_{1/8} \leq 1/2$ on $x_n=r_0$, we get
$$u(\tilde  x)= W^\eps_{\bar s }(\tilde x) < V^\eps (\tilde x) + \alpha c_0 \eps $$
which contradicts \eqref{u-p>cep}.

\end{proof}

\begin{cor} \label{corollary}Let $u$ be as in Lemma $\ref{HI}$. Then  the modulus of continuity of $\tilde
u^\eps$ is H\"older outside $[0, \sigma(\eps)]$ with $\sigma(\eps) \to 0$ as $\eps \to 0.$
\end{cor}

\section{The dichotomy}\label{dicoD}

Throughout this section, $u$ is a (Lipschitz) solution to \eqref{fbtp}.

We introduce the class of functions $\mathcal Q_{f_-}$ defined as
$$Q_{p,q,\omega, M} = (x \cdot \omega - \frac 1 2 x^T M x)(p+ q \cdot x) - \frac 1 2 (f_-(0)+p tr M) (x \cdot \omega)^2,$$
with $p \in \R, q \in \R^n, M \in S^{n\times n},$ such that 
$$q \cdot \omega =0 , \quad M \omega = 0, \quad \|M\|\leq 1.$$

In the degenerate case, we use these functions to approximate $u^-$ in a $C^{2,\gamma}$ fashion at a smaller and smaller scale. The goal is to reach a scale $\rho$ where $u^-$ is trapped between two translations of $Q$ of size $\rho^{2+\gamma^*}.$ This would guarantee that the full $u$ is $(V, \rho^{2+\gamma^*}, \bar \delta)$ flat at that scale, which allows us to apply the improvement of flatness result of the non-degenerate case.

Initially $u^+$ is $(V, \bar\lambda^{2+\gamma}, \bar \delta)$ flat while $u^-$ is $C^{2,\gamma}$ close to the configuration $Q_{0,0, e_n, 0}$. This closeness improves at a $C^{2,\gamma}$ rate until (possibly) the slope $p$ of the approximating polynomial $Q=Q_{p,q,\omega,M}$ is no longer zero, say at scale $\lambda.$ However, to obtain the desired flatness of $u$, we need to reach a scale $\rho=\lambda r$ for $r << \lambda^{1/\gamma}$. It is necessary to exploit also the information that the flatness of $u^+$ is in fact improving at a $C^2$ rate for a little while,  hence allowing us to continue the iteration on the negative side and obtain that $u^-$ is $C^{2,\gamma}$ close to a configuration $Q$ at an even smaller scale. As already pointed out, in the case of the $C^{1,\gamma}$ estimates of \cite{DFS} this issue was not present, as in the degenerate case it was sufficient to reach a scale $\rho= \lambda r$ with $r=\lambda^{\gamma/2}.$ This makes the $C^{2,\gamma}$ proof much more sophisticated and technically involved.

We are ready to start developing the tools to apply the strategy described above.

The next lemma relates the closeness of $u^-$ to a function in the class $\mathcal Q_{f_-}$ with a one-phase free boundary condition for $u^+_\nu$ of the type in Proposition \ref{IFtp}-(ii) (one-phase improvement of flatness). The condition is satisfied in the viscosity sense of Remark \ref{normals_v}.
The proof relies on a variant of the pointwise $C^{1,\gamma}$ estimate, which we describe in  Appendix A. 
We also, need the following easy remark.
\begin{rem}\label{Qsize} We remark that
$$|Q_{p,q,\omega, M}- Q_{p,q,\omega,0}|\leq C|p|r^2+|q|r^3 \quad \text{in $B_r$};$$
$$|Q_{p,q,e_n,0}-Q_{p, \tilde q, \omega, 0}|\leq (|p|r+ (2|q|+|f_-(0)|)r^2)|e_n - \omega| \quad \text{in $B_r$};$$ where $\tilde q$ is a rotation of $q$ by the angle between $e_n$ and $\omega.$
\end{rem}

Let $V= V^{\alpha}_{M, e_n, a_n}$ and $Q= Q_{p, q, e_n, M}$ with $\alpha=G(|p|).$

\begin{lem}\label{fbcondition} Let $u^+$ be $(V, r^2 \lambda^{2+\gamma}, \bar \delta)$ flat in $B_{r \lambda},$
$$|f_-(x) - f_-(0)| \leq \bar \delta |x|^\gamma, \quad \|f_-\|_{\infty} \leq \bar \delta
$$
and 
$$|u^- - Q| \leq \bar \delta^{1/2} (r \lambda)^{2+\gamma}, \quad \text{in $B^-_{r\lambda}(u)$,}$$  for $r \leq 1$, with $p \leq 0$, $|p| \sim(\bar \delta^{1/2} \lambda^{1+\gamma}), |q| =O(\bar \delta^{1/2}\lambda^\gamma),$ and
\be\label{bound_r}\bar \delta^{1/2} r^\gamma \geq  2 \lambda^{1+\gamma}.\ee
Then,
\be\label{link1} |u^+_\nu - \alpha| \leq  \bar C \bar \delta^{1/2} r \lambda^{1+\gamma} \quad \text{on $F(u) \cap B_{r\lambda}$}\ee
in the viscosity sense, with $\bar C$ universal.
\end{lem}
\begin{proof} 
 Let $\mathcal P(x)$ be a quadratic polynomial and let $x_n=\mathcal P(x)$ touch $F(u)$ by the negative side at $x_0 \in F(u) \cap B_{r\lambda/2}$, with the coefficients of $\mathcal P$ of size 1. Let $\nu$ be the normal to $F(u)$ at $x_0$ (pointing toward the positive phase). Then, we want to show that 
\be
u_\nu^+(x_0) \geq \alpha -\bar \delta^{1/2} r \lambda^{1+\gamma}
\ee
Denote $\nu^-=-\nu.$
Since (for some $t$), $$u_\nu^+(x_0) = G(u_{\nu}^-(x_0)) = G(|p|) + G'(t) (\nabla u^-(x_0) \cdot \nu^- - |p|),$$
it suffices to show that
\be\label{link_again} \nabla u^-(x_0) \cdot \nu \leq p+ C \bar \delta^{1/2} r \lambda^{1+\gamma}.\ee

Let $$u_{r\lambda}(x):= u(r\lambda x), \quad Q_{r\lambda}(x):=Q(r\lambda x), \quad \mathcal P_{r\lambda}(x):=\mathcal P(r\lambda x),$$  and set \be\label{defv} v(x): =  \bar \delta^{-1/2}(r\lambda)^{-(2+\gamma)}(u^-_{r\lambda} - Q_{r\lambda})(x), \quad x \in B_1,\ee
Then,
\be\label{propv}|v| \leq 1, \quad |\Delta v| \leq 2\bar \delta^{1/2} \quad \text{in $B_1^-(u_{r\lambda})$}.\ee
Moreover, since $r$ satisfies \eqref{bound_r}, using the estimates for $|p|, |q|$ and the flatness of the free boundary, we obtain that
$$|v| \leq \bar \delta^{1/2} \quad \text{on $F(u_{r\lambda})$.}$$
\smallskip

We claim that $$v_{\nu} (y_0) \leq Cr^{-\gamma}, \quad y_0=\frac{x_0}{r\lambda}.$$

To prove the claim, we wish to apply Theorem \ref{refine} in the Appendix to $-v$. Indeed, assume $F(u_{r\lambda})$ is the graph of a function $g$ in the $\nu^-$ direction. Then, since $x_n= \frac{1}{r\lambda} \mathcal P_{r\lambda}$ touches $F(u_{r\lambda})$ at $y_0$ by the negative side we conclude that
$$g \leq C (r\lambda) |x-y_0|^2.$$
Thus, it suffices to prove that at the point $y_0$, for any $\nu^\perp$ perpendicular to $\nu^-$ we have
\begin{eqnarray}\label{grad}&|\nabla \bar Q_{r\lambda} \cdot \nu^\perp| = O(r^{-\gamma}), \quad -O(r^{-\gamma}) \leq \nabla \bar Q_{r\lambda} \cdot \nu^- \leq  O(\frac{r^{-\gamma}}{r\lambda})\\
\label{grad2}&|D^2\bar Q_{r\lambda}(\nu, \nu^\perp)|= O(r^{-\gamma}), \quad |D^2 \bar Q_{r\lambda}(\nu^-,\nu^-)| = O(\frac{r^{-\gamma}}{r\lambda}),\end{eqnarray}
where we denoted 
$$\bar Q_{r\lambda}= \bar \delta^{-1/2} (r\lambda)^{-(2+\gamma)}Q_{r\lambda}.$$
Indeed, it is easy to check that in $B_\rho \cap \{|x_n| < \rho^2\}$ we have:
$$\nabla Q = p e_n + O(|q|\rho + |p|\|M\|\rho + |f_-(0)|\rho^2+ |p||tr M| \rho^2 + |q|\|M\|\rho^2)$$
and
$$D^2 Q = -p M  - 2(\frac 1 2 f_-(0)+ p tr M) e_n \otimes e_n + O(|q|+\|M\| |q| \rho).$$
In particular, for $\rho=r\lambda$, using the bounds for $|p|, |q|, |f_-(0)|$ we conclude that
$$\nabla Q= p e_n + O(\bar \delta^{1/2} r\lambda^{1+\gamma})$$
and
$$D^2Q = -(f_-(0)+ 2p tr M) e_n \otimes e_n + O(\bar \delta^{1/2} \lambda^\gamma).$$
Thus,  we easily obtain the second estimate in \eqref{grad}-\eqref{grad2} (recall that $p\leq 0$ and $e_n \cdot \nu^- \leq 0$). The first one follows by using that at $y_0$ we have $|\nu - e_n| \leq \lambda r$.

Hence the claim holds and rescaling back we get
\be\label{normals}(\nabla u^-(x_0) - \nabla Q(x_0)) \cdot \nu \leq C \bar \delta^{1/2}r\lambda^{1+\gamma} . \ee
Moreover, at  such point $|\nu -e_n|\leq \lambda r$ hence we have
\begin{eqnarray*}& |\nabla Q(x_0) \cdot \nu - p| \leq |\nabla Q(x_0)\cdot \nu - \nabla Q(x_0) \cdot e_n| + |Q_{n}(x_0)-p| \\
&\leq \|\nabla Q\|_\infty |\nu -e_n| + O(\bar \delta^{1/2} r \lambda^{1+\gamma})
\leq Cr \bar \delta^{1/2}  \lambda^{1+\gamma}\end{eqnarray*}
and we reach the desired conclusion.
\end{proof}

In the next proposition we show that if we are in a degenerate setting, that is $u^-$ is very close to the configuration $Q_{0,0,e_n,0}$, then either this is preserved at a smaller (universal) scale or $u^-$ becomes close to a configuration $Q_{p,q,{\bf e}, M}$ with a non-zero slope $p$. In either case the positive part $u^+$ also improves. Without loss of generality, we still denote the universal constants below as in previous propositions.

\begin{prop}\label{first_step}There exist universal constants $\bar \lambda, \bar \delta, \bar \eta$ such that if
\begin{equation}\label{flat1tp_dic_again}
\text{$u^+$ is $(V, \lambda^{2+\gamma}, \bar \delta)$ flat in $B_{\lambda}, \lambda \leq \bar \lambda$}
\end{equation} with  $V= V^1_{M,e_n, a_n} \in \mathcal V_{f_+},$ 
\begin{equation}\label{flat2tp_dic_again}
|f_\pm(x) - f_\pm(0)| \leq \bar \delta |x|^\gamma, \quad \|f_-\|_{\infty} \leq \bar \delta
\end{equation}
and 
\be\label{more_again} |u^- - Q_{0,0, e_n, 0}| \leq \bar \delta^{1/2} \lambda^{2+\gamma}, \quad \text{in $B^-_\lambda(u)$}\ee
then either one of the following holds:
\begin{enumerate}
\item there exists $\bar V= V^1_{\bar M, \bar{\bf e}, \bar a_{\bar{\bf e}}}  \in \mathcal V_{f_+},$ such that \be\label{automatic} \text{$u^+$ is $(\bar V,  (\bar \eta \lambda)^{2+\gamma}, \bar \delta)$ flat in $B_{\bar \eta \lambda},$} \ee
and
\be\label{conclusion}|u^- - Q_{0,0, \bar{\bf e}, 0}| \leq \bar \delta^{1/2} (\bar \eta \lambda)^{2+\gamma}, \quad \text{in $B^-_{\bar \eta \lambda}(u)$;}\ee
\item there exists $V^*= V^{\alpha^*}_{M^*, {\bf e}^*, a^*_{{\bf e}^*}}  \in \mathcal V_{f_+},$ such that
 $$\text{$u^+$ is $(V^*,\bar \eta^2 \lambda^{2+\gamma}, \bar \delta)$ flat in $B_{\bar \eta \lambda},$} $$
and
$$|u^- - Q_{p^*,q^*, {\bf e}^*,  M^*}| \leq \bar \delta^{1/2} (\bar \eta \lambda)^{2+\gamma}, \quad \text{in $B^-_{\bar \eta \lambda}(u)$,}$$ for $\alpha^*= G(|p^*|)$ and $p^*<0, |p^*| \sim(\bar \delta^{1/2} \lambda^{1+\gamma}), |q^*| =O(\bar \delta^{1/2}\lambda^\gamma).$

\end{enumerate}

\end{prop}

\begin{proof} All the universal constants will be specified throughout the proof. In particular, for $\bar \eta$ fixed, $\bar \lambda <<\bar \delta << \bar \eta$ and they are small enough so that Proposition \ref{IF} can be applied. 
 
In view of Lemma \ref{fbcondition} and Proposition \ref{IF} (see also Remark \ref{normals}), then \eqref{automatic} holds.
 
 Now, for $x \in B_1$, set
$$\tilde u(x)= \frac 1 \lambda u(\lambda x), \quad \tilde f_-(x)= \lambda f_-(\lambda x)$$
$$v(x)= \bar \delta^{-1/2}\lambda^{-(1+\gamma)}(\tilde u^-(x) + \frac 1 2 \tilde f_-(0)x_n^2).$$
Since $\tilde u^- \geq 0$ and $|\tilde f_-(0)| \leq \lambda \bar \delta$
\be \label{almost_positive}
v \geq - \lambda^{-\gamma}\frac{\bar \delta^{1/2}}{2} x_n^2.
\ee

If we prove that 
\be\label{WTS} |v| \leq \frac 1 2 \bar \eta^{2+\gamma}, \quad \text{in $B_{\bar\eta}^-(\tilde u),$}\ee
then in view of Remark \ref{Qsize}, we can conclude that \eqref{conclusion} holds as well, by choosing $\bar \delta$ small enough (depending on $\bar \eta$.)

From assumptions  \eqref{flat1tp_dic_again}-\eqref{flat2tp_dic_again}-\eqref{more_again} we get ($\lambda$ small)
$$F(\tilde u) \subset \{-\lambda \leq x_n \leq \lambda\}= S_\lambda,$$
$$|v| \leq 1 \quad \text{in $B_1^-(\tilde u)$}, \quad |\Delta v| \leq \bar \delta^{1/2} \quad \text{in $B_1^-(\tilde u),$} \quad |v|\leq \delta^{1/2}\lambda \quad \text{on $F(\tilde u)$}.$$

Using a barrier, we can prove that
\be \label{strip_1}|v| \leq C \lambda \quad \text{in $S_\lambda \cap B_{2/3}^-(\tilde u).$}\ee
Indeed, let $\xi$ satisfy  $$\Delta \xi= -\delta^{1/2} \quad \text{in $B_1 \cap \{x_n < \lambda\}$}$$ with $$\xi=0 \quad \text{on $x_n=\lambda,$} \quad \xi= 1 \quad \text{on $\p B_1\cap \{x_n <\lambda\}$}.$$
Then, by the maximum principle,
\be\label{max}\xi + \delta^{1/2} \lambda \geq v \quad \text{in $B_1^-(\tilde u)$}\ee
and
$$\xi \leq C|x_n - \lambda| \quad \text{in $B_{2/3} \cap \{x_n < \lambda\}$},$$
from which the desired upper bound follows. The lower bound is proved similarly.

Thus, for $\lambda$ small (depending on $\bar \eta$),  \eqref{WTS} holds in $S_\lambda$.

We now analyze what happens in $B_{\bar \eta} \cap \{x_n < -\lambda\}.$
Denote by $w$ the solution to 
$$\Delta w=0 \quad \text{in $B^-:=B_{2/3} \cap \{x_n <0\}$}$$
with boundary data
$$ w=v \quad \text{on $\p B^- \cap \{x_n <0\}$}, \quad w=0 \quad \text{on $\{x_n=0\}$}.$$

Notice that, $$|w| \leq C \lambda \quad \text{on $x_n=-\lambda$.}$$ This can be obtained  using the barrier $\xi$ above and \eqref{max} (and the corresponding upper bound for $v$).
Thus, by the maximum principle
$$|w-v| \leq C (\lambda + \bar \delta^{1/2}) \quad \text{in $B_{1/2} \cap \{x_n \leq -\lambda\}$}$$

In particular, if 
\be\label{boundonw} |w| \leq \frac 1 4 \bar \eta^{2+\gamma}, \quad \text{in $B_{\bar \eta} \cap \{x_n \leq -\lambda\}$},\ee
then \eqref{WTS} holds as long as $\lambda$ and $\bar \delta$ are small enough (depending on $\bar \eta$.)

We now determine under which conditions \eqref{boundonw} is valid. First, expanding $w$ near 0, we have
\be\label{exp_1}w(x)= x_n(p+ q \cdot x + O(|x|^2)), \quad q\cdot e_n=0, \quad |p|, |q| \leq C.\ee

Thus, there exists $\bar \eta$ universal, such that if 
\be \label{ab}|p| \leq \bar \eta^2, \quad |q| \leq \bar \eta\ee
then \eqref{boundonw} holds.

We now prove that:
$$p \geq - \bar\eta^4 \Rightarrow \eqref{boundonw} \ \text{holds} \ \Rightarrow (i)$$
$$p < - \bar \eta^4 \Rightarrow (ii).$$

First, we observe that
\be\label{claim_d} |v-w| \leq C(\lambda + \delta^{1/2}|x_n|), \quad \text{in $B_{1/2} \cap \{x_n \leq -\lambda\}$}.\ee Indeed, let $\psi, \phi$ solve
$$\Delta \psi= 0, \Delta \phi= -1 \quad \text{in $B_{2/3}\cap \{x_n <-\lambda\}$}$$
with
$$\psi=\phi=0 \quad \text{on $\p B_{2/3} \cap \{x_n < -\lambda\}$}$$
$$\psi=C \lambda, \phi = 0 \quad \text{on $B_{2/3} \cap \{x_n=-\lambda\}.$}$$
Then,
$$v- w \leq \psi + \delta^{1/2} \phi \quad \text{in $B_{2/3}\cap \{x_n \leq -\lambda\}$}$$
and the desired upper bound follows. The lower bounds is obtained similarly.

\smallskip

We now distinguish three cases.

\smallskip

\noindent{\it Case 1}. $|p| \leq \bar \eta^4.$

In this case we show that $|q| \leq \bar \eta$ hence \eqref{ab} is satisfied.

Indeed, assume $|q| > \bar \eta$, say to fix ideas  $|q_1|>\frac{1}{\sqrt n}\bar \eta$. Let $\bar x= ((sign q_1) \bar\eta^{2},0,\ldots, -\lambda^\beta)$ with $\beta=(1+\gamma)/2$. Then, using \eqref{almost_positive}-\eqref{claim_d} we get 
$$- \frac 1 2 \lambda \bar \delta^{1/2} \leq v(\bar x) \leq w(\bar x)+ C \lambda + C \bar \delta^{1/2} \lambda^{\beta}$$
from which, using \eqref{exp_1}, we deduce that (for $\bar \lambda, \bar \delta<<\bar \eta$)
$$C \lambda \geq \lambda^\beta(p+|q_1| \bar \eta^2 - \bar C \bar \eta^4 - C \bar \delta^{1/2}) \geq C \bar \delta^{1/2} \lambda^\beta$$
and we reach a contradiction if $\bar \lambda^{1-\gamma} << \bar \delta.$
 
\smallskip

\noindent {\it Case 2.}  $p>\bar \eta^4$.

In this case we can argue similarly as in Case 1, by choosing $\bar x = (0, -\lambda^\beta).$ 

\smallskip

\noindent{\it Case 3.} $p< - \bar \eta^{4}.$

In this case, we first notice that in view of \eqref{strip_1}-\eqref{claim_d} and the linear growth of $w$ (extended to zero in $x_n >0$), we have
$$|v-w| \leq C (\lambda+ \bar \delta^{1/2}) \quad \text{in $B_{\bar \eta}^-(\tilde u)$}.$$
Moreover, in view of \eqref{exp_1} we also have that ($B_{\bar \eta}^-(\tilde u) \subset \{x_n < \lambda\}$)
$$|w- x_n(p+q \cdot x)| \leq C \bar \eta^3 \quad \text{in $B_{\bar \eta}^-(\tilde u)$}.$$
(for $\lambda$ small.)

Combining these two inequalities and using the formula for $v$ and rescaling back, we obtain $$|u^- - Q_{\bar p, \bar q, e_n, 0}| \leq \frac 1 2 \bar \delta^{1/2}(\bar \eta \lambda)^{2+\gamma} \quad \text{in $B_{\bar \eta \lambda}^-(u)$}$$
with
$$\bar p = \bar \delta^{1/2} \lambda^{1+\gamma} p, \quad \bar q=  \bar \delta^{1/2} \lambda^{\gamma} q.$$
Thus, in view of Remark \ref{Qsize}, if $\bar{\bf e}, \bar M$ are given by \eqref{automatic} (which we have already observed to hold), then
$$|u^- - Q_{\bar p,  q^*, \bar{\bf e}, \bar M}| \leq \bar \delta^{1/2}(\bar \eta \lambda)^{2+\gamma} \quad \text{in $B_{\bar \eta \lambda}^-(u)$}, \quad |\bar q- q^*| \leq C \lambda^{1+\gamma}.$$
Finally, setting $\alpha^*= G(|\bar p|)$, and letting $V^*= V^{\alpha^*}_{\bar M, \bar {\bf e}, a^*_{\bar{\bf e}}}$ we obtain the desired conclusion in $(ii)$ (again for $\bar \delta << \bar \eta$). Here $a^*$ is chosen so that $V^* \in \mathcal V_{f_+},$ i.e.
$$2\alpha^* a^*\cdot \bar{\bf e} -\alpha^* tr\bar M=f_{+}(0).$$
 Thus, the claim follows from the fact that $$|\alpha^* -1|=O(\delta^{1/2}\lambda^{1+\gamma}), \quad |a^*-\bar a| = O(|\alpha^*-1||2\bar a_{\bar{\bf e}} - tr \bar M|),$$ with
 $$r|\bar a|, r\|M\| \leq \bar \delta r^{\frac{1+\gamma}{2}}, \quad r=\bar \eta \lambda,$$
  and one can easily check that 
$$|V^{1}_{\bar M, \bar{\bf e}, \bar a_{\bar{\bf e}}}- V^{\alpha^*}_{\bar M, \bar{\bf{e}}, a^*_{\bar{\bf e}}}| \leq C(|\alpha^*-1|r + |a^* - \bar a|r^2)\leq  \frac 1 2 \bar \eta^2 \lambda^{2+\gamma} \quad \text{in $B_{r}$.}$$
\end{proof}

\section{The iteration.}\label{iteration7}

In the next proposition we show that if we are in an ``intermediate degenerate setting", that is if at a small enough scale $u^+$is flat and $u^-$ is close to a configuration $Q$ with a small non-zero slope, then the flatness of $u^+$ improves for a little while at a $C^2$ rate while the estimate on $u^-$ improves at a $C^{2,\gamma}$ rate.

Let $V= V^{\alpha}_{M, e_n, a_n} \in \mathcal V_{f_+}$ and $Q= Q_{p, q, e_n, M}$ with $\alpha=G(|p|).$

\begin{prop}\label{it}There exist universal constants $\bar \lambda, \bar \delta, \bar \eta,$ such that if
\begin{equation}\label{f1}
 \text{$u^+$ is $(V, r^2\lambda^{2+\gamma}, \bar \delta)$ flat in $B_{r\lambda}, \lambda \leq \bar \lambda$}
\end{equation}  for some $r\leq \bar \eta$ with $\bar \delta^{1/2} r^\gamma \geq  2 \lambda^{1+\gamma},$ and 
\begin{equation}\label{f2}
|f_\pm(x) - f_\pm(0)| \leq \bar \delta |x|^\gamma, \quad \|f_-\|_{\infty} \leq \bar \delta
\end{equation}
\be\label{f3}|u^- - Q| \leq \bar \delta^{1/2} (r \lambda)^{2+\gamma}, \quad \text{in $B^-_{r \lambda}(u)$,}\ee  with $p <0, |p| \sim \bar \delta^{1/2} \lambda^{1+\gamma}, |q|= O(\bar \delta^{1/2}\lambda^\gamma),$
then
\begin{equation}\label{option2}
 \text{$u^+$ is $(\bar V,  (\bar \eta r)^2\lambda^{2+\gamma}, \bar \delta)$ flat in $B_{r\bar \eta \lambda}$,}
\end{equation}
and 
\be\label{more4} |u^- - \bar Q| \leq \bar \delta^{1/2}(r\bar \eta \lambda)^{2+\gamma}, \quad \text{in $B^-_{\lambda r \bar \eta}(u)$}\ee
with $\bar V= V^{\bar \alpha}_{\bar M, \bar {\bf e}, \bar a_{\bar{\bf e}}} \in \mathcal V_{f_+}$, $\bar Q= Q_{\bar p, \bar q, \bar{\bf e}, \bar M}$,  $\bar \alpha= G(|\bar p|)$, $\bar p <0, |\bar p| \sim \bar \delta^{1/2} \lambda^{1+\gamma}, |\bar q|= O(\bar \delta^{1/2}\lambda^\gamma).$
\end{prop}

\begin{proof} 

Call,
$$\tilde u(x)= \frac{1}{r\lambda} u(r\lambda x), \quad \tilde V(x)=\frac{1}{r\lambda} V(r\lambda x),  \quad \tilde Q(x)=\frac{1}{r\lambda} Q(r\lambda x)\quad x \in B_1.$$

{\bf Step 1.} As usual, universal constants are small enough so that previous results can be applied. They will be made smaller in the proof, with $\bar \lambda<< \bar \delta << \bar \eta.$ 
Let,
$$\eps = r \lambda^{1+\gamma}.$$ 
In view of Lemma \ref{fbcondition}, $\tilde u/\alpha$ satisfies the assumptions of Proposition \ref{IFtp} (see also Remark \ref{normals_v}) hence
\begin{equation}\label{important_1}
 \text{$\tilde u^+$ is $(\tilde V^*, \bar \eta^{2+\gamma}\eps, \bar \delta)$ flat in $B_{\bar \eta}$}
\end{equation}
with $\tilde V^*=V^\alpha_{M^*, \bar {\bf e}, a^*_n}$.
Let $\bar M= M^*/(r\lambda), \bar a= a^*/(r\lambda).$

\smallskip

{\bf Step 2.}  Let $$v(x)= \bar \delta^{-1/2}(r\lambda)^{-(1+\gamma)}(\tilde u^-(x) - \tilde Q(x)).$$

We argue similarly as in Proposition \ref{first_step}.

From assumptions \eqref{f1}-\eqref{f2}-\eqref{f3} and the estimates for the sizes of $p, q, r$ we get
$$F(\tilde u) \subset \{-r\lambda \leq x_n \leq r\lambda\}= S_{r\lambda},$$
$$|v| \leq 1 \quad \text{in $B_1^-(\tilde u)$}, \quad |\Delta v| \leq \bar 2\delta^{1/2} \quad \text{in $B_1^-(\tilde u),$} \quad |v|\leq \bar \delta^{1/2}\quad \text{on $F(\tilde u)$}.$$

Using a barrier, we can prove that
\be \label{strip}|v| \leq C \lambda+\bar \delta^{1/2} \quad \text{in $S_{r\lambda} \cap B_{2/3}^-(\tilde u).$}\ee 

Indeed, let $\xi$ satisfy $$\Delta \xi= -\delta^{1/2} \quad \text{in $B_1 \cap \{x_n <r \lambda\}$}$$ with $$\xi=0 \quad \text{on $x_n=r\lambda,$} \quad \xi= 1 \quad \text{on $\p B_1\cap \{x_n <r\lambda\}$}.$$
Then, by the maximum principle,
\be\label{max2}\xi + \bar \delta^{1/2} \geq v \quad \text{in $B_1^-(\tilde u)$}\ee
and
$$\xi \leq C|x_n - r\lambda| \quad \text{in $B_{2/3} \cap \{x_n < r\lambda\}$},$$
from which the desired upper bound follows. The lower bound is proved similarly.

We now analyze the region where $x_n \leq -r\lambda.$

Denote by $w$ the solution to 
$$\Delta w=0 \quad \text{in $B^+:=B_{2/3} \cap \{x_n < 0\}$}$$
with boundary data
$$ w=v \quad \text{on $\p B^+ \cap \{x_n < 0\}$}, \quad w=0 \quad \text{on $\{x_n=0\}$}.$$

Using the barrier $\xi$ above and \eqref{max2} we get that
$$|w| \leq C \lambda + \bar \delta^{1/2} \quad \text{on $B_{2/3} \cap \{x_n = -r\lambda\}.$}$$

Hence by the maximum principle, ($\lambda << \bar \delta$)
\be|v - w| \leq C \bar \delta^{1/2} \quad \text{on  $B_{2/3} \cap \{x_n < -r\lambda\}.$}\ee

In fact, in view of \eqref{strip} and the linear growth of $w$ we get (say $w=0$ on $x_n >0$)
\be\label{key_2}|v - w| \leq C \bar \delta^{1/2} \quad \text{on  $B_{1/2}^-(\tilde u).$}\ee

We now expand $w$ near 0, i.e.
\be\label{exp}w(x)= x_n(p^*+ q^* \cdot x + O(|x|^2)), \quad q^*\cdot e_n=0, \quad |p^*|, |q^*| \leq C.\ee
Hence, for $\bar \eta$ small universal, and $\lambda << \bar \eta$, ( $B_{\bar \eta}^-(\tilde u) \subset \{x_n < r\lambda\}$)
\be\label{key_1}|w-x_n(p^*+ q^* \cdot x )| \leq \frac 1 2 \bar \eta^{2+\gamma} \quad \text{in $B_{\bar \eta}^-(\tilde u).$}\ee

In this case, using the formula for $v$ and rescaling back, we obtain from \eqref{key_2}-\eqref{key_1} that
$$|u^- - Q_{\bar p, \bar q, e_n, M}| \leq \frac 4 5 \bar \delta^{1/2}(\bar \eta r \lambda)^{2+\gamma} \quad \text{in $B_{\bar \eta r \lambda}^-(u)$},$$
with
$$\bar p = \bar \delta^{1/2} (r\lambda)^{1+\gamma} p^*+p, \quad \bar q=  \bar \delta^{1/2} (r\lambda)^{\gamma} q^*+q.$$
Using that $r \leq \bar \eta$, it easily follow that for $\bar \eta$ small universal,
$$\bar p <0, \quad |\bar p|\sim \bar \delta^{1/2} \lambda^{1+\gamma}$$
and clearly
$$|\bar q| = O(\bar \delta^{1/2}\lambda^\gamma).$$
If we replace $M$ with $\bar M$, then for $\bar \delta << \bar \eta$, the error $E$ has size
$$E= O(\|M-\bar M\||x|^2(|\bar p| + |\bar q||x|))\leq \frac{1}{10} \bar \delta^{1/2} (\bar \eta r \lambda)^{2+\gamma}$$
where in the last inequality we used that $\lambda^{1+\gamma} \leq \frac 1 2 \bar \delta^{1/2} r^\gamma.$
Similarly, if we now replace $e_n$ with $\bar{\bf e}$ and $\bar q$ with the corresponding $\tilde q$ such that $\tilde q \cdot \bar{\bf e} =0$ we get an error $E$ of size
$$E=O(|x||e_n -\bar{\bf e}| (|\bar p|+|\bar q||x|) + (|f_-(0)|+|\bar p|\|\bar M\|)|x|^2|e_n-\bar{\bf e}|^2 + |\bar q-\tilde q|(|x|+ \|M\||x^2|))$$
and again 
$$E\leq \frac{1}{10} \bar \delta^{1/2} (\bar \eta r \lambda)^{2+\gamma}$$
using that $\lambda^{1+\gamma} \leq \frac 1 2 \bar \delta^{1/2} r^\gamma.$
Thus, 
$$|u^- - Q_{\bar p, \tilde q, \bar{\bf e}, \bar M}| \leq \bar \delta^{1/2} (\bar \eta r \lambda)^{2+\gamma}$$
and $|\tilde q|=O(\bar \delta^{1/2}\lambda^\gamma).$
Finally, let $\bar \alpha = G(|\bar p|).$ Then,
$$|\alpha-\bar \alpha| = O(|p-\bar p|) = O(\bar \delta^{1/2}(r\lambda)^{1+\gamma}).$$
Thus, dropping the dependence on the subscripts $\bar a, \bar{\bf e}, \bar M$
$$|V^{\alpha} - V^{\bar \alpha}|\leq C|\alpha-\bar \alpha| \bar \eta r \lambda \leq C \bar \delta^{1/2}(r\lambda)^{2+\gamma}\bar \eta $$
and for $\bar \eta$ small universal and $\bar \delta << \bar \eta$
$$V^{\bar \alpha}(x+\bar \eta^{2+\gamma} r^2 \lambda^{2+\gamma}\bar{\bf e})+  C \bar \delta^{1/2}(r\lambda)^{2+\gamma}\bar \eta \leq V^{\bar \alpha}(x+\frac 1 2\bar \eta^{2} r^2 \lambda^{2+\gamma}\bar{\bf e}).$$
Hence, scaling back, we conclude from \eqref{important_1} that \eqref{option2} holds for $\bar V= V^{\bar \alpha}_{\bar M, \bar{\bf e}, \bar a_{\bar{\bf e}}}$ (arguing similarly for the lower bound.) As at the end of the proof of Proposition \ref{first_step} we can now modify $\bar a$ slightly to guarantee that $\bar V \in \mathcal V_{f_+}.$
\end{proof}

Finally, in the proposition below, we show that after reaching a small enough scale, the approximation of $u^-$ with a configuration $Q$ is good enough to recover the full $C^{2,\gamma^*}$ flatness of $u$ (in a non-degenerate setting.)

\begin{prop}\label{together}
There exist $\bar \lambda, \bar \delta, \gamma^*$ universal such that if 
\begin{equation}\label{flat1tp_final}
\text{$u^+$ is $(V, r^2\lambda^{2+\gamma}, \bar \delta)$ flat in $B_{r\lambda}, \lambda \leq \bar \lambda$}
\end{equation}  with $V=V^{\alpha}_{M,e_n, a_n}$, for $r$ with $\bar \delta^{1/2}r^\gamma \in [2\bar \eta^{\gamma}\lambda^{1+\gamma},  2\lambda^{1+\gamma})$ and 
\be\label{more_final}|u^- - Q_{p,q, e_n,  M}| \leq \bar \delta^{1/2} (r \lambda)^{2+\gamma}, \quad \text{in $B^-_{r \lambda}(u)$,}\ee for $\alpha= G(|p|)$ and $p <0, |p| \sim \bar \delta^{1/2} \lambda^{1+\gamma}, |q|= O(\bar\delta^{1/2}\lambda^\gamma),$
then
\be  \text{$u$ is $(\bar V, (r\lambda)^{2+\gamma^*}, \bar \delta)$ flat in $B_{r\lambda}$}\ee
with $\bar V=V^{\alpha,\beta}_{M,e_n, a,b} \in \mathcal V_{f_\pm, G}, \beta=|p|$. \end{prop}

\begin{proof} Let $\bar \lambda, \bar \delta$ be the constants in Proposition \ref{it} , with $\bar \lambda << \bar \delta$ to be made possibly smaller. Let $\gamma^*$ be given, to be specified later.

Call $$W_\beta:= \beta (1+b \cdot x) (x_n - \frac 1 2 x^T M x)$$
with
$$\beta=|p|, \quad b' = \frac{1}{|p|} q', \quad 2\beta b_n = \beta tr M + f_-(0).$$
Then it follows from \eqref{more_final} that
\be\label{more_final_2}|u - W_\beta| \leq C \bar \delta^{1/2} (r \lambda)^{2+\gamma}, \quad \text{in $B^-_{r \lambda}(u)$.}\ee
Hence,
\be\label{final_beta} W_\beta(x - (r\lambda)^{2+\gamma^*}e_n) \leq u \leq W_\beta(x + (r\lambda)^{2+\gamma^*}e_n)\quad \text{in $B^-_{r \lambda}(u)$,}\ee
as long as
\be\label{r1} r \leq C \lambda^{\frac{1+\gamma^*}{\gamma -\gamma^*}}. \ee
Moreover, call
$$\eps := (r\lambda)^{1+\gamma^*},$$ then 
\be (r\lambda)|b'| \leq \bar \delta^2 \quad (r\lambda)^2|b_n|\|M\| \leq \bar \delta^2 \eps \ee
as long as
\be\label{r2} r \leq C \lambda^{\frac{\gamma+\gamma^*}{1-\gamma^*}}. \ee

Now, let 
$$W_\alpha:=\alpha(1+ a \cdot x)(x_n - \frac 12 x^T M x)^+$$
with 
$$\alpha a':= \beta G'(\beta) b'.$$ Then, $\alpha a' = O(|q'|)$ and in view of \eqref{flat1tp_final} we get
\be W_\alpha(x - Cr^2\lambda^{2+\gamma}e_n) \leq u^+(x) \leq W_\alpha(x + Cr^2\lambda^{2+\gamma}e_n) \quad \text{in $B_{r\lambda}$} \ee and conclude that
\be\label{final_+} W_\alpha(x - (r\lambda)^{2+\gamma^*}e_n) \leq u^+(x) \leq W_\alpha(x + (r\lambda)^{2+\gamma^*}e_n) \quad \text{in $B_{r\lambda}$}\ee 
as long as 
\be\label{r3} r \geq C \lambda^{\frac{\gamma}{\gamma^*}-1}.\ee
Notice that all bounds on $r$ are satisfied as long as 
\be\label{gamma*}\gamma^* < \frac{\gamma^2}{1+2\gamma}\ee and $\bar \lambda$ is small enough.
Now the conclusion follows combining \eqref{final_beta} and \eqref{final_+}.
\end{proof}

\medskip

We conclude this section by exhibiting the proof of our main Theorem \ref{flatmain1}.

\smallskip

{\it Proof of Theorem $\ref{flatmain1}.$}
According to Lemma \ref{first_step_1}, after rescaling, we can assume that $u$ satisfies either the assumptions of Proposition \ref{first_step} or Proposition \ref{IF2tp} with $\lambda=\bar \lambda$ (say $0 \in F(u)$). In the latter case, we can apply Proposition \ref{IF2tp} indefinitely. 
If $u$ falls in the degenerate case of Proposition \ref{first_step}, then either we can iterate the conclusion $(i)$ indefinitely or we denote by $\lambda^*= \bar \eta^{k} \bar \lambda$ the first value for which $(ii)$ holds i.e, without loss of generality, 
 $$\text{$u^+$ is $(V,\bar \eta^2 {\lambda^*}^{2+\gamma}, \bar \delta)$ flat in $B_{\bar \eta \lambda^*},$} $$
for some $V=V^{\alpha, \beta}_{M,e_n, a, b} \in \mathcal V_{f_\pm, G}$
$$|u^- - Q_{p,q, e_n,  M}| \leq \bar \delta^{1/2} (\bar \eta \lambda^*)^{2+\gamma}, \quad \text{in $B^-_{\bar \eta \lambda^*}(u)$,}$$ for $\alpha= G(|p|)$ and $p<0, |p| \sim(\bar \delta^{1/2} {\lambda^*}^{1+\gamma}), |q| =O(\bar \delta^{1/2}{\lambda^*}^\gamma).$

We now follow under the assumptions of Proposition \ref{it} or possibly Proposition \ref{together}, with $r=\bar \eta$, $\lambda=\lambda^*$. We apply the conclusion of Proposition \ref{it} till the first $\bar r=\bar \eta^{m_0}$ (possibly $m_0=1$) for which $\bar \delta^{1/2} \bar r^\gamma \in [2\bar \eta^\gamma {\lambda^*}^{1+\gamma}, 2 {\lambda^*}^{1+\gamma}).$ Then we conclude by Proposition \ref{together} that 
\be  \text{$u$ is $(\bar V, (\bar r\lambda)^{2+\gamma^*}, \bar \delta)$ flat in $B_{\bar r\lambda}$}\ee
with $\bar V=V^{\alpha,\beta}_{M,e_n, a,b} \in \mathcal V_{f_\pm, G}, \beta=|p|$ and we can apply indefinitely Proposition \ref{IF2tp}. To guarantee the $C^{2,\gamma^*}$ improvement, we have to check that as $r$ decreases from $\bar \eta$ to $\bar r$ we have
$${\lambda^*}^{2+\gamma} r^2 \leq (\lambda^* r)^{2+\gamma^*}.$$
Thus we need, $$\bar r \geq (\lambda^*)^{\frac{\gamma-\gamma^*}{\gamma^*}},$$
which follows from the fact that (see \eqref{gamma*}), $$\gamma^* < \frac{\gamma^2}{1+2\gamma}.$$

\section{Appendix A} \label{appendixa}

In this short section we recall standard pointwise $C^{1,\alpha}$ estimates for solutions to elliptic equations in $C^{1,\alpha}$ domains (see for example \cite{MW} for further details.). We also presents a few variants which are needed in the previous section. 

Let $$\Omega: = \{x_n > g(x')\} \cap B_1 \quad \Gamma:= \{x_n=g(x')\} \cap B_1, \quad g(0)=0, \quad \nabla_{x'} g (0)=0$$ with $g \in C^{1,\alpha}$ and say \be |g| \leq |x|^{1+\alpha}.\ee Let $u$ be a bounded solution to 
\be\label{point} \begin{cases} 
\Delta u =f \quad \text{in $\Omega,$}\\
u=\varphi \quad \text{on $\Gamma$,}
\end{cases}\ee
for $f \in L^{\infty}(\Omega)$, $\varphi \in C^{1,\alpha}(\Gamma).$
\begin{thm}\label{pointwise}
Assume that
\be\label{var} |\varphi(x) - l(x')|  \leq |x|^{1+\alpha}\ee with $l$ a linear function.
If
\be \|l\|_{\infty}, \|u\|_{\infty}, \|f\|_\infty \leq 1 \quad \text{in $\Omega$},\ee
then $u$ is $C^{1,\alpha}$ at 0 and
\be
|\nabla u (0)| \leq C
\ee
with $C=C(n,\alpha).$
\end{thm}
It follows also that$$\p_i u (0)=\p_i l \quad i \neq n.$$

\begin{rem}\label{basic}It easy to see that if \eqref{var} is replaced by
\be \varphi(x) - l(x') \leq |x|^{1+\alpha} \quad (\text{rsp. $\geq |x|^{1+\alpha}$})\ee
then the following conclusion holds: $\p_n u(0)$ exists and 
$$ \p_n u(0) \leq C.$$ Indeed we can apply Theorem \ref{pointwise} to the function $v$ which solves problem \eqref{point} with $\varphi$ replaced by  $l(x') + |x|^{1+\alpha}.$ By the maximum principle $u \leq v$, and the conclusion follows. \end{rem}

We need the following refinement of this remark. Let $\varphi$ be defined in $B_1$.

\begin{thm}\label{refine}
Assume that \be\label{g} -|x|^{1+\alpha} \leq g \leq \sigma|x|^{1+\alpha}\ee for some small $\sigma>0.$ If
$\|u\|_\infty, \|f\|_{\infty} \leq 1,$ and $\varphi \in C^2$ satisfies
\be |\varphi(0)|, |\p_i \varphi(0)| \leq 1 \quad i \neq n\ee
\be -1 \leq \p_n\varphi(0) \leq \frac{1}{\sigma}\ee
\be |\p_{ij} \varphi | \leq 1 \quad \text{in $B_1$}, \quad (i,j)\neq (n,n)\ee and
\be |\p_{nn} \varphi| \leq \frac{1}{\|g\|_\infty}\ee
then $\p_n u(0)$ exists and 
$$\p_n u(0) \leq C.$$
\end{thm}

To obtain this estimate it suffices to apply the Remark \ref{basic} with 
$$l(x')= \varphi(0)+ \sum_{i \neq n} \p_i \varphi(0) x_i.$$
Indeed
$$\varphi(x) - l(x') \leq \p_n\varphi(0) x_n + C |x|^2 +  \frac{1}{\|g\|_\infty}x_n^2 \leq C |x|^{1+\alpha}.$$

A similar statement holds when $g$ satisfies the inequality \be g \geq - \sigma |x|^{1+\alpha}.\ee

\section{Appendix B} \label{appendixB}

For the reader convenience we recall the technique of \cite{KNS} to 
transform a general (possibly nonlinear) two-phase free boundary problem into an elliptic system with coercive
boundary conditions. 

Let $u$ be a classical solution to a two-phase free boundary problem governed by a second order elliptic equation, say in $B_1$ with $0 \in F(u)$.
For $\sigma $ small, the
partial hodograph map 
\begin{equation*}
y^{\prime }=x^{\prime },\quad y_{n}=u^{+}(x)
\end{equation*}%
is $1-1$ from $\overline{B_{1}^{+}}(u)\cap B_{\sigma }\left( 0\right) $ onto
a neighborhood of the origin $U\subset \{y_{n}\geq 0\},$ and flattens $F(u)$
into a set $\Sigma \subset \{y_{n}=0\}.$ The inverse mapping is the partial
Legendre transformation%
\begin{equation*}
x^{\prime }=y^{\prime },\quad x_{n}=\psi (y),
\end{equation*}%
where $\psi $ satisfies $y_{n}=u^{+}\left( y^{\prime },\psi \left( y\right)
\right) ,$ $y\in U$. The free boundary is now the graph of $x_{n}=\psi \left(
y^{\prime },0\right) $. 

Concerning the negative part, let $C$ be a constant larger than  $\partial_{y_n}\psi$ in $U.$
Introduce the reflection map 
\begin{equation*}
x^{\prime }=y^{\prime },\quad x_{n}=\psi (y)-Cy_{n},
\end{equation*}%
which is $1-1$ from a neighborhood of the origin $U_{1}\subseteq U$ onto $%
\overline{B_{1}^{-}}(u)\cap B_{\sigma }\left( 0\right) $ (choosing $\sigma $
smaller, if necessary). Define in $U_{1}$ 
\begin{equation*}
\phi (y)=u^{-}(y^{\prime },\psi (y)-Cy_{n}).
\end{equation*}%

It is easily seen that the $x$ derivatives of $u^\pm$ are expressed in terms of the $y$ derivatives of $\psi$ and $\phi.$ Since $u$ is a solution to a two-phase problem, it follows that $\psi$ and $\phi$ solve in $U_{1}$ a nonlinear system of the type

\begin{equation}  \label{sis1}
\left\{ 
\begin{array}{l}
\mathcal{F}_{1}(D^2\psi, \nabla \psi, \psi, y)=0\\ 
\mathcal{F}_{2}(D^2\phi ,D^2\psi, \nabla \phi, \nabla \psi, \psi, y)=0%
\end{array}%
\right.
\end{equation}

Moreover  the free boundary conditions 
\begin{equation*}
u^{+}=u^{-}\quad \text{and}\quad |\nabla u^{+}|=G(|\nabla u^{-}|),\quad 
\text{\ on }F(u)
\end{equation*}
become (for an appropriate $\tilde G$)

\begin{equation}
\left\{ 
\begin{array}{l}
\phi (y^{\prime })=0 \quad \text{on $y_n=0,$}\\ 
\\ 
{\partial _{y_{n}}\psi}=\tilde G \left(
\partial _{y_{n}}\phi, \nabla _{y^{\prime }} \psi \right) \quad \text{on $y_n=0.$}
\end{array}%
\right.  \label{sis2}
\end{equation}%

Since $|\nabla u^{+}|>0$ on $F\left( u\right) $ and $G$ is strictly
increasing, the system (\ref{sis1}) is elliptic with coercive boundary
conditions (\ref{sis2}), under the natural choice of weights (see \cite{KNS}%
, p. 94, 95). 

In the particular case when the equation governing the problem is in divergence form, then \eqref{sis1} will also be in divergence form. On the other hand, if the system  (\ref{sis1})-(\ref{sis2})  has no special structure, then higher regularity  follows by classical results on elliptic-coercive systems in \cite{ADN, Morrey}, as long as 
$u$ is in $C^{2,\alpha}(\overline{B_{1}^{+}(u)})\cap C^{2,\alpha}(\overline{%
B_{1}^{-}(u)}).$



\begin{thebibliography}{9999}
\bibitem[ADN]{ADN} S. Agmon, A. Douglis, L. Nirenberg, {\it Estimates near the
boundary for solutions of elliptic partial differential equations satisfying
general boundary conditions. I.,} Comm. Pure Appl. Math. 12, 1959, 623--727.
\bibitem[B1]{B1} Batchelor G.K., {\it On steady laminar flow with closed streamlines at large Reynolds number}, J. Fluid.Mech. \textbf{1} (1956), 177--190.


\bibitem[DFS]{DFS} De Silva D., Ferrari F., Salsa S., \textit{Two-phase
problems with distributed source: regularity of the free boundary,} Anal.
PDE 7 (2014), no. 2, 267--310.

\bibitem[DFS2]{DFS2} De Silva D., Ferrari F., Salsa S.,
\textit{Perron's solutions for two-phase free boundary problems with distributed sources,} Nonlinear Anal. 121 (2015), 382--402.

\bibitem[DFS3]{DFS3} De Silva D., Ferrari F., Salsa S. \textit{Free boundary regularity for fully nonlinear non-homogeneous two-phase problems,} J. Math. Pures Appl. (9) 103 (2015), no. 3, 658--694.
\bibitem[E]{E} Engelstein M., {\it A two phase free boundary problem for the harmonic measure,}  arXiv:1409.4460.
\bibitem[EM]{EM} Elcrat A.R, Miller K.G., {\it Variational formulas on Lipschitz domains}, Trans. Amer. Math. Soc. \textbf{347} (1995), 2669--2678.
\bibitem[FL]{FL} Friedman A., Liu Y. {\it A free boundary problem arising in magnetoydrodynamic system}, Ann. Scuola Norm. Sup. Pisa, Cl. Sci. \textbf{22} (1994), 375--448.
\bibitem[KNS]{KNS} Kinderlehrer D., Nirenberg L., Spruck J., {\it Regularity in
elliptic free-boundary problems I,} Journal d'Analyse Mathematique, Vol 34
(1978) 86-118.
\bibitem[K]{K} Koch H., {\it Classical solutions to phase transition problems are smooth,} Comm. Partial Differential Equations 23 (1998), no. 3-4, 389--437.
\bibitem[KL]{KL} Kriventsov D., Lin F-H., {\it Regularity for shape optimizers: the non degenerate case}, arXiv:1609.02624. 
\bibitem[LW]{LW} Lederman C., Wolanski N.,  {\it A two phase elliptic singular perturbation problem with a forcing term,} J. Math. Pures Appl. (9) {\bf 86} (2006), no. 6, 552--589.
\bibitem[M]{Morrey}  Morrey C. B. Jr., {Multiple integrals in the calculus of variations,}
Reprint of the 1966 edition Classics in Mathematics. Springer-Verlag, Berlin, 2008. x+506 pp. 

\bibitem[MW]{MW}Ma, F., Wang, L.
{\it Boundary first order derivative estimates for fully nonlinear elliptic equations}.  
J. Differential Equations {\bf 252} (2012), no. 2, 988--1002. 


\end{thebibliography}
\end{document}